%% file: first_publication_compiled.tex
\documentclass[a4paper,11pt]{article}
\usepackage{MyStyle}

\usepackage{leftindex}

\usepackage{algorithm}
\floatname{algorithm}{Method}
\usepackage{algorithmic}

\newcolumntype{C}{>{\(}c<{\)}}
\usepackage{planarforest}
\input{planarforest_tikz.tex}

\newcommand{\EF}{EF}
\newcommand{\EFv}{\mathcal{EF}}

\allowdisplaybreaks%
\newtheorem{theorem}{Theorem}[section]
\newtheorem{definition}[theorem]{Definition}
\newtheorem*{definition*}{Definition}

\newtheorem{proposition}[theorem]{Proposition}
\newtheorem{lemma}[theorem]{Lemma}
\newtheorem{remark}[theorem]{Remark}
\newtheorem*{remark*}{Remark}

\newtheorem*{remarks*}{Remarks}
\newtheorem{corollary}[theorem]{Corollary}
\newtheorem{assumption}[theorem]{Assumption}
\newtheorem*{notation*}{Notation}
\newtheorem{ex}[theorem]{Example}
\newtheorem*{ex*}{Example}
\newtheorem{exs}[theorem]{Examples}
\newtheorem*{exs*}{Examples}

\newtheorem*{app*}{Application}
\newtheorem{conjecture*}{Conjecture}

\def\ts{\thinspace}

\usepackage{shuffle}
\def\shuffle{{\sqcup\mathchoice{\mkern-4.8mu}{\mkern-4.5mu}{\mkern-4.5mu}{\mkern-4.5mu}\sqcup}}

\title{Post-processed frozen-flow methods for the long time sampling of ergodic dynamics on Riemannian manifolds}
\author{ Adrien Busnot Laurent\textsuperscript{1} and Sébastien Macé\textsuperscript{1}}

\usepackage[english]{babel}
\usepackage{geometry}
\usepackage{csquotes}
\usepackage{hyperref}
\usepackage{mathtools}
\usepackage{mathrsfs}
\usepackage{tikz}
\usetikzlibrary{arrows.meta, calc}
\usepackage{xcolor,pagecolor}
\usepackage{booktabs,tabularx}
\usepackage[colorinlistoftodos]{todonotes}
\usepackage{pgfplots}
\pgfplotsset{compat=1.18}

\DeclareMathOperator{\IBP}{IBP}
\DeclareMathOperator{\Irred}{I_{R}}
\DeclareMathOperator{\Irredv}{\mathcal{I}_\mathcal{R}}
\DeclareMathOperator{\RED}{RED}
\newcommand{\deriv}{\mathrm{d}}
\newcommand{\ps}[2]{ \left\langle#1 \vert#2 \right\rangle}
\newcommand{\moy}[1]{ \left\langle#1 \right\rangle}
\newcommand{\crochet}[2]{ \left[ #1, #2 \right]}
\newcommand{\norm}[1]{\left\Vert#1 \right\Vert}
\newcommand{\suite}[2]{\left(#1\right)_{#2 \in\mathds{N}}}

\newcommand{\inte}[1]{ \displaystyle{ \int_{#1} } }
\newcommand{\espe}[1]{\mathds{E} \left[ #1 \right]}
\newcommand{\restr}[2]{\mathchoice{\setbox1\hbox{${\displaystyle #1}_{\scriptstyle #2}$}%
 \restrictionaux{#1}{#2}}
 {\setbox1\hbox{${\textstyle #1}_{\scriptstyle #2}$}%
 \restrictionaux{#1}{#2}}
 {\setbox1\hbox{${\scriptstyle #1}_{\scriptscriptstyle #2}$}%
 \restrictionaux{#1}{#2}}
 {\setbox1\hbox{${\scriptscriptstyle #1}_{\scriptscriptstyle #2}$}%
 \restrictionaux{#1}{#2}}}
\newcommand{\restrictionaux}[2]{{#1\,\smash{\vrule height .8\ht1 depth .85\dp1}}_{\,#2}}

\begin{document}
\footnotetext[1]{
 Univ Rennes, INRIA (Research team MINGuS), IRMAR (CNRS UMR 6625) and ENS Rennes, France.
 Adrien.Busnot-Laurent@inria.fr,
 Sebastien.Mace@univ-rennes.fr.}
\maketitle

\begin{abstract}
In this work, we propose a novel intrinsic approach to the approximation of ergodic SDEs on Riemannian manifolds, which include Riemannian Langevin dynamics.
In opposition to the standard extrinsic approaches such as penalization methods and projection methods, our methodology does not use embeddings or coordinates and only relies on natural geometric operations: geodesics, parallel transport,\ldots
We give a criterion for high order of accuracy for the invariant measure, develop new intrinsic numerical methods designed solely for sampling the invariant measure, and derive high order conditions using a new algebraic operation on exotic Lie-Butcher series.
In the spirit of the Leimkuhler-Matthews method, our approach prioritizes long time sampling efficiency over finite time accuracy, and outperforms the previous extrinsic and intrinsic approaches in terms of cost for a given accuracy, which we illustrate with several numerical experiments.

\end{abstract}
\section{Introduction}\label{section:Introduction}

The aim of this paper is the design of high-order sampling methods of ergodic stochastic dynamics on Riemannian manifolds.
More precisely, let \((\mathcal{M},g)\) be a smooth, complete, connected  $D$-dimensional Riemannian manifold endowed with a metric \(g\) and let \(\nabla\) be its associated Levi-Civita connection. Let \(E_1,\dots,E_D\) be a global orthonormal frame basis (for simplicity), that is, for all \(x\) in \(\mathcal{M}\), the set \(\{E_1(x),\dots,E_D(x)\}\) is an orthonormal basis of the tangent space \(T_x \mathcal{M}\). 
Given a smooth and Lipschitz vector field \(F : x \mapsto \underset{d=1}{\overset{D}{\sum}} f^d(x) E_d(x)\), we consider stochastic differential equations on $\MM$ of the following form:
\begin{equation}\label{equation:SDE} \deriv X(t) = F(X(t))\deriv t + \sqrt{2}\sum_{d=1}^D E_d(X(t)) \circ \deriv W_d(t) , \quad X(0) = X_0 \in \mathcal{M}. \end{equation}
Equation~\eqref{equation:SDE} can be understood as a SDE with additive noise, as it rewrites (see~\cite{Hsu02sao}) as
\[
\deriv X(t) = (F+\nabla_{E_n}E_n)(X(t))\deriv t + \sqrt{2} \deriv B_\mathcal{M}(t),
\]
where \(B_\MM(t)\) is a Brownian motion on \(\MM\) and we use the Einstein summation notation.
The class of SDEs~\eqref{equation:SDE} includes the celebrated Riemannian Langevin dynamics
\begin{equation}\label{equation:Langevin} \deriv X(t) = -\nabla V(X(t))\deriv t + \sqrt{2}\deriv B_\mathcal{M}(t),
\end{equation}
when $F$ derives from a potential, that is, when for a smooth function $V\colon \MM\rightarrow\R$, one has
\begin{equation}\label{equation:correction}
 F = -\nabla V - \nabla_{E_n}E_n, \quad f^d = - E_d[V] - \ps{\nabla_{E_n}E_n}{E_d}.
\end{equation}

Under growth assumptions on $F$ (see, for instance, the Bakry-Emery criterion~\cite{Bakry86cne} for the Langevin case~\eqref{equation:Langevin}), the flow \(X(t)\) of~\eqref{equation:SDE} is ergodic, that is, its long-time behaviour is described by a deterministic measure \(\deriv \mu_\infty=\rho_\infty \deriv \vol_\MM\), that is absolutely continuous with respect to the Riemannian volume form \(\deriv \vol_\MM\), in the sense that for a large class of smooth functions \(\varphi\) and for all initial condition \(X_0 \in \MM\), one has
\[ \underset{T \to \infty}{\lim} \frac{1}{T} \int_0^T \varphi(X(t))\deriv t = \inte{\mathcal{M}} \varphi(y) \deriv \mu_\infty(y) \text{ almost surely.}\]
In the case of Riemannian Langevin dynamics~\eqref{equation:Langevin}, the invariant measure is given by the Gibbs density $\rho_\infty\propto e^{-V}$.

In this paper, we design new numerical methods that have a high order of accuracy for sampling the invariant measure of~\eqref{equation:SDE}.
We consider one-step numerical integrators \( {(X_n)}_{n = 0,\dots,N} \), on a discretization \(t_n = nh\) of a time interval \([0,T]\) for a time step \(h=T/N\), given by a random perturbation of the identity:
\begin{equation}\label{equation:method}
 X_{n+1} = \Psi_h(X_n), \quad X_0 \in \MM.
\end{equation}
We also consider post-processed methods
\begin{equation}\label{equation:post-process}
 X_0 \in \MM, \quad X_{n+1} = \Psi_h(X_n), \quad \overline{X}_N = \overline{\Psi}_h(X_N),
\end{equation}
where the post-processor \(X \mapsto \overline{X}\) is a random perturbation of the identity map on \(\MM\), that is applied only once at the very end of the trajectory.
Assuming ergodicity of the numerical method~\eqref{equation:method} for a measure \( \deriv \mu^h_\infty=\rho^h_\infty \deriv \vol_\MM\),
we propose methods of high order of accuracy \(p > 1\) for the invariant measure of~\eqref{equation:SDE}, that is, that satisfy for $h$ small enough,
\[
\abs{\inte{\mathcal{M}} \varphi \deriv \mu^h_\infty - \inte{\mathcal{M}} \varphi \deriv \mu_\infty} \leq C h^p.
\]
A crucial difficulty of the approach is that the integrators have to evolve on $\MM$ in order for the measure $\deriv \mu^h_\infty$ to be absolutely continuous w.r.t.\ts$\deriv \vol_\MM$.
Moreover, our approach focuses on \emph{intrinsic} integrators, that are, numerical methods that evolve on $\MM$ and that do not depend on an embedding of $\MM$ in a higher-dimensional Euclidean space or on a choice of coordinates.
Among the existing \emph{extrinsic} approaches, we mention the popular projection methods (see~\cite{Lelievre10fec, Laurent21ocf} and references therein), which are conveniently implemented by embedding $\MM$ in a vector space of higher dimension, but often face severe timestep restrictions.
We also mention the existing intrinsic approaches for SDEs on manifolds~\cite{Malham08slg, Muniz22hso, Bharath23sae, Luesink26stf}, that either focus on order one of accuracy or strong convergence, or are defined on specific manifolds only (i.e., Lie groups).
Our approach relies on the class of frozen-flow methods recently introduced in~\cite{Bronasco25hoi}, on which we add post-processors~\cite{Vilmart15pif}. This class of methods can be seen as a stochastic extension of the Crouch-Grossman and commutator-free Lie group methods~\cite{Crouch93nio, Owren99rkm, Iserles00lgm, Celledoni03cfl, Owren06ocf}.
In opposition to~\cite{Malham08slg, Muniz22hso}, our approach generalises the Lie group methods to any smooth manifold and focuses only on the approximation for the invariant measure.
An important feature of the new integrators is that they do not rely on the use of embeddings or on local coordinates. Their formulation, convergence analysis, order theory, and implementation are entirely intrinsic.

To derive methods of high order for the invariant measure, it is sufficient to use methods of weak order $p$ as they will have at least order $p$ for the invariant measure~\cite{Talay90eot}. However, there exists many efficient methods of low weak order and high order for the invariant measure in a variety of (Euclidean) contexts~\cite{BouRabee10lra, Leimkuhler13rco, Abdulle15lta, Leimkuhler16tco, Bronasco25els}.
Among these, the celebrated Leimkuhler-Matthews method~\cite{Leimkuhler13rco} is of weak order one and second order for the invariant measure for solving~\eqref{equation:SDE} when $\MM=\R^D$:
\begin{align*}
H_n = X_n+ \frac{\sqrt{2h}}{2} \xi^d_n E_d,\quad
X_{n+1} = X_n+ h F(H_n) + \frac{\sqrt{2h}}{2}\xi^d_n E_d,\quad
\overline{X_N} = X_N +\frac{\sqrt{2h}}{2} \overline{\xi}^d_n E_d.
\end{align*}
Our intrinsic approach with frozen-flow methods allows us to derive in particular a straightforward generalisation of the Leimkuhler-Matthews method to the Riemannian setting,

where only one evaluation of the vector field $F$ is needed (see Section~\ref{sub:High order frozen methods for sampling ergodic dynamics} for the details).
We recall that, in comparison, the extrinsic approach with projection methods from~\cite{Laurent21ocf} requires four evaluations of $F$, has a technical implementation, and faces stability issues.

The derivation of the order conditions for the invariant measure relies on exotic Lie-Butcher series and on the Weitzenböck connection algebra.
The exotic Lie-Butcher series were introduced in~\cite{Bronasco25hoi} for deriving weak high-order estimates. They generalise naturally the exotic Butcher series~\cite{Laurent20eab, Laurent21ata, Bronasco22ebs, Bronasco22cef} and the Lie-Butcher series~\cite{Iserles00lgm, Owren06ocf, MuntheKaas08oth}.
While such formalism is necessary for dealing with the intricate calculations of order conditions, we emphasize that the algebraic objects considered here are interesting beyond their numerical use.
We mention in particular the works~\cite{Bronasco22cef, Laurent23tue, Bronasco25hoi} that study the geometric and algebraic properties of exotic forests and the works~\cite{MuntheKaas13opl, Ebrahimi15otl, AlKaabi22aao, Grong23pla, MuntheKaas23lat} that draw strong links between planar trees and computations in specific connection algebras.
The numerical integrators considered here rely on elementary geometric operations (geodesic, parallel transport) for a specific curvature-free connection, called the Weitzenböck connection. This choice allows us to use planar trees for the calculations, to which we add the exotic feature for representing the Laplace-Beltrami operator in terms of forests.
Thanks to this algebraic formalisation, we provide a methodology for the design of intrinsic integrators of any high-order for the invariant measure.
More precisely, we define the integration by parts of planar exotic forests, thus generalising~\cite{Laurent20eab, Bronasco22ebs, Bronasco22cef} to the manifold setting, and we describe the order conditions for the invariant measure by a character on a modified shuffle algebra of exotic forests.
Contrary to the Euclidean case where such theory naturally leads to modified equations of any order and stochastic backward error analysis for the invariant measure, there are profound algebraic differences on manifolds that make intrinsic stochastic backward error analysis challenging.

The paper is organised the following way.
We present in Section~\ref{section:Preliminaries} the main results of the paper. After presenting the notation and main assumptions, we give a convenient characterisation of the high-order for the invariant measure, generalising the Euclidean works~\cite{Debussche12wbe, Abdulle14hon, Vilmart15pif}. Then, we apply this criterion to derive new simple methods of second order for the invariant measure.
The proof of the criterion for high order for the invariant measure is presented in Section~\ref{section:High order analysis for the invariant measure}.
In Section~\ref{Section:Intrinsic order conditions with planar exotic forests}, we apply the exotic Lie-Butcher series formalism for the derivation of order conditions for the invariant measure, and we extend the integration by parts on exotic forests to the planar context.
We present numerical experiments on a variety of classical manifolds in Section~\ref{section:Numerical experiments} and we present outlooks and future works in Section~\ref{section:conclusion}.

\section{Preliminaries, characterization of the invariant measure and new numerical methods}\label{section:Preliminaries}

\subsection{Notation and main assumptions}\label{sub:Notations and main assumptions}

For \(v \in T\MM\), let \(\abs{v}=\sqrt{g(v,v)}\) be its Riemannian norm. For
\(o \in \MM\). Then we denote by \(r = d(o,\cdot)\) the Riemannian distance
map. It is 1-Lipschitz and \(r^2\) is smooth on \(\MM \setminus Cut_o\).
For all \(\varphi \in C^\infty\), \(\deriv \varphi\) is its differential and
\(\abs{\deriv \varphi}\) its Riemannian norm. We denote by \(v[\varphi](x)\) the differential \(\deriv \varphi(x) \cdot v\) of
\(\varphi\) in the direction of \(v \in T_x \MM\) at the point \(x \in \MM\). Let
denote the set of test function by \(C_P^\infty(\MM)\), the set of smooth functions
whose derivatives of all orders have polynomial growth and satisfy estimates of
the form
\[ \abs{E_{d_q}[\dots E_{d_1}[\varphi]\dots]}(x) \leq C(1+{r(x)}^K) , \quad \text{ for } q=0,1,\dots \]

Let \(F : x \mapsto \underset{d=1}{\overset{D}{\sum}} f^d(x) E_d(x)\) be a
vector field and its unique decomposition in the orthonormal frame, then \(F\)
is an element of \( \mathfrak{X}_P(\mathcal{M}) \) if \(f^d \in C_P^\infty(\MM)
\) and if its components are Lipschitz, that is, \(\abs{\deriv f^d} \leq C \).
Equipped with the Jacobi bracket ${[-,-]}_J$, the space $(\mathfrak{X}_P(\mathcal{M}),{[-,-]}_J)$ is a Lie algebra.

\begin{definition}
 The Weitzenböck affine connection \(\vartriangleright : \mathfrak{X}_P(\mathcal{M}) \times \mathfrak{X}_P(\mathcal{M}) \to \mathfrak{X}_P(\mathcal{M})\), is given by
 \[(X \vartriangleright Y) = (X  [Y^d]) E_d. \]
 We extend this notation to functions by $X \vartriangleright \varphi=X[\varphi]$.
\end{definition}

\begin{remark}
 The choice of the Weitzenböck connection was already implicitly made in~\cite{Owren06ocf}. It can be shown that its associated curvature vanishes. If one further assumes that the torsion $T$ is constant, then $(\mathfrak{X}_P(\mathcal{M}),\vartriangleright, -T)$ is a post-Lie algebra~\cite{Grong23pla}.
\end{remark}

In the spirit of~\cite{Lundervold11hao, Grong23pla, Bronasco25hoi}, let the differential operators on $\MM$ be given by the tensor algebra $T(\mathfrak{X}_P(\mathcal{M}))$ whose product is denoted by $\cdot$.
Differential operators act on functions \(\varphi\) by:
\[ (X_1 \cdots X_n)\vartriangleright \varphi = x^{i_1}_1 \cdots x^{i_n}_n E_{i_1} [ \ldots E_{i_n}[\varphi] \ldots ],\]
where $X_j=\sum_{i_j=1}^D x^{i_j}_j E_{i_j}$ is the unique decomposition of the vector field $X_j$ in the frame basis.
\begin{remark}
 The Laplace-Beltrami operator is expressed by
 \[\Delta \varphi = \underset{d=1}{\overset{D}{\sum}} \left( E_d\cdot E_d - \nabla_{E_d}E_d \right)\vartriangleright\varphi .\]
\end{remark}
The product $\vartriangleright$ extends to $T(\mathfrak{X}_P(\mathcal{M}))$ by the so-called Guin-Oudom process~\cite{Oudom08otl, Ebrahimi14tme}, which extends straightforwardly in a curvature-free setting (see~\cite{Bronasco25hoi}).
The space of differential operators $(T(\mathfrak{X}_P(\mathcal{M})),\cdot, \Delta_\shuffle)$ is a Hopf algebra when equipped with the deshuffle coproduct~\cite{Lundervold11hao, Grong23pla}. In the context of a connection with constant torsion, $(T(\mathfrak{X}_P(\mathcal{M})),\cdot, \Delta_\shuffle, \vartriangleright)$ yields a post-Hopf algebra~\cite{Li23pha}. Note that taking into account the action of differential operators on functions would naturally yield post-Hopf algebroid structures~\cite{Bronasco25hoi, Busnot25pha}.

The generator \(\mathcal{L}\) of the SDE~\eqref{equation:SDE} is defined by
\begin{equation}\label{equation:generator}
 \mathcal{L} \varphi = \underset{d=1}{\overset{D}{\sum}} (f^d E_d + (E_d\cdot E_d)) \vartriangleright\varphi.
\end{equation}
In the case of the Langevin equation~\eqref{equation:Langevin}, the generator is expressed as
\begin{equation}\label{equation:generator Langevin}
 \mathcal{L} \varphi = - \nabla V [\varphi] + \Delta \varphi.
\end{equation}
To ensure that the equation~\eqref{equation:SDE} is well-posed and that the
solution does not blow up, we shall assume the following.
\begin{assumption}\label{assumption:2.1}
 The vector fields \(E_1,\dots,E_D\) are smooth and bounded.
 The vector field \(F\) belongs to \(\mathfrak{X}_P(\mathcal{M})\).
 The generator satisfies
 \begin{equation}\label{equation:elliptic}
 \mathcal{L}r^2 \leq \nu + \lambda r^2
 \end{equation}
 on $\mathcal{M}\setminus \mathrm{Cut}_o$ for some constants $\lambda \in\mathbb{R}$ and $\nu \geq 1$.
\end{assumption}

There exist various criteria which ensure that
inequality~\eqref{equation:elliptic} is satisfied. The first criterion is
compactness, which can be used for classical manifolds such as the sphere
\(\mathds{S}^n\) and the Lie group \(SO_q(\mathds{R})\). This condition being
seldom satisfied, a second classical criterion is the Bakry-Émery criterion
from~\cite{Bakry86cne}. This criterion applies to the case of
equation~\eqref{equation:SDE}, where the generator writes as~\eqref{equation:generator Langevin}. One says that the potential \(V\) satisfies
the criterion if there exists \(\kappa \in \mathds{R}\) such that
\[\Ric + \Hess (V) \geq \kappa.\]
The term \(\Ric\) designs the Ricci tensor on \((\mathcal{M},g)\) and is
defined as the trace of the Riemann tensor. Note that this handy criterion has
been adapted in~\cite{Antonyuk07nonexp} for equation~\eqref{equation:SDE} with
a lower bound of the operator \(\Ric-\nabla F\). The moment conditions
from~\cite{Li94stochdiff} implies stochastic completeness.

Under Assumption~\ref{assumption:2.1}, the equation~\eqref{equation:SDE}
generates a Markovian semigroup on \(C_P^\infty(\MM)\). For all \(\varphi \in
C_P^\infty(\MM) \), the function \( u : (t,x) \mapsto \espe{\varphi(X(t))\vert
 X(0) = x} \) satisfies the Kolmogorov equation
\begin{equation}\label{equation:Kolmogorov}
 \partial_t u = \mathcal{L}(u), \quad u(0,x) = \varphi(x).\end{equation}
To the contrary to~\cite{Bronasco25hoi}, we are interested in the
behaviour in long time of the solution of~\eqref{equation:SDE}.
\begin{definition}
 A process \(X\) is ergodic if there exists \(\deriv \mu_\infty\) a unique invariant measure with density function \(\rho_\infty\) with respect to \(\deriv \vol_\MM\), such that for all \(\varphi \in C_P^\infty(\MM)\) and for all initial condition \(X(0) \in \MM\), it follows that
 \begin{equation}
 \underset{T \to \infty}{\lim} \frac{1}{T} \underset{0}{\overset{T}{\int}} \varphi(X(s))\deriv s = \inte{\mathcal{M}} \varphi(y) \deriv \mu_\infty(y) \text{ almost surely,}
 \end{equation}
 where \(\deriv \vol_\mathcal{M}\) is the Riemannian measure on \(\MM\) defined in~\cite{Lee19irm}.
\end{definition}
\begin{definition}
 A numerical method~\eqref{equation:method} is ergodic if there exists \(\deriv \mu^h\) a unique invariant measure with density function \(\rho^h\) with respect to \(\deriv \vol_\MM\), such that for all \(\varphi \in C_P^\infty(\MM)\) and for all initial condition \(X(0) \in \MM\), it follows that
 \begin{equation}
 \underset{N \to \infty}{\lim} \frac{1}{N+1} \underset{n=0}{\overset{N}{\sum}} \varphi(X_n) = \inte{\mathcal{M}} \varphi(y) \rho^h(y) \deriv \vol_\mathcal{M}(y) \text{ almost surely.}
 \end{equation}
 The method is of order \(p \geq 1\) for the invariant measure if there exists \(h_0 > 0\) and \(C >0\) which depends on \(h_0\) and \(\varphi\) such that for all \(h \in (0,h_0)\),
 \begin{equation}\label{equation:error}
 \abs{e(\varphi,h)} \leq C h^p \text{ with } e(\varphi,h) = \underset{N \to \infty}{\lim} \frac{1}{N+1} \underset{n=0}{\overset{N}{\sum}} \varphi(X_n) - \inte{\mathcal{M}} \varphi(y) \deriv \mu_\infty(y).
 \end{equation}
\end{definition}
In the spirit of~\cite{Abdulle14hon}, the ergodicity of the stochastic process is a prerequisite of our analysis.
\begin{assumption}\label{assumption:H1}
 Under Assumption~\ref{assumption:2.1}, there exists a unique invariant measure \(\deriv \mu_\infty\) with density \(\rho_\infty\) with respect to \(\deriv \vol_\MM\) is the unique solution of the equation
 \begin{equation}
 \mathcal{L}^*\deriv \mu_\infty = 0,
 \end{equation}
 with \(\mathcal{L}^*\) the adjoint of \(\mathcal{L}\) in \(L^2(\mathcal{M})\).
\end{assumption}
In the case of the Riemannian Langevin equation~\eqref{equation:Langevin},
Assumption~\ref{assumption:H1} is automatically satisfied and the density
\(\rho_\infty\) is explicitly given by the Gibbs measure:
\begin{equation}\label{equation:Gibbs}
 \rho_\infty = e^{-V} / Z, \quad Z = \inte{\MM} e^{-V} \deriv \vol_\MM.
\end{equation}
\begin{assumption}\label{assumption:H2}
 For all function \(g \in C_P^\infty(\MM)\) with zero mean on \(\MM\), there exists a unique function \(\mu \in C_P^\infty(\MM)\) such that \( \mathcal{L}^* \mu = g\) and \( \inte{\mathcal{M}} \mu(y) \deriv \mu_\infty(y) = 0 \).
\end{assumption}
The following assumption from~\cite{Bronasco25hoi}, automatically satisfied on
\(\mathds{R}^d\), ensures the regularity of the semigroup~\eqref{equation:Kolmogorov}.
\begin{assumption}\label{assumption:2.3}
 For all function \(\varphi \in C_p^\infty(\mathcal{M})\), the map
 \[u : (t,x) \mapsto \espe{\varphi(X(t))\vert X_0 =x}\]
 belongs to \(C^\infty((0,T),C_P^\infty(\mathcal{M}))\), that is, for all \(k \geq 0\) there exist constants \(C > 0\) and \(\kappa \geq 0\) such that for all \(x \in \MM\)
 \[ \underset{t \in (0,T)}{\sup} \abs{\partial^k_t u(t,x)} \leq C(1+r^\kappa_w(x))\abs{\varphi}_{C^{2N+2}}. \]
\end{assumption}
Hence, taking the limit as \(t \to \infty\) of the time integration of Equation~\eqref{equation:Kolmogorov} leads to, under
Assumption~\ref{assumption:2.3},
\begin{equation}\label{equation:lim_u}
 \underset{t \to \infty}{\lim} u(t,x) = \varphi(x) + \int_0^\infty \mathcal{L} u(s,x) \deriv s.
\end{equation}
The following assumption is an ergodicity condition concerning convergence to
the invariant measure. We introduce the following norm and seminorm:
\[ \norm{\psi}_{C^k} = \underset{\underset{0 \leq \abs{j}\leq k}{j=(j_1,\dots,j_D)}}{\sup} \underset{x \in \mathcal{M}}{\sup} \abs{\partial_j \psi(x)} \text{ and } \abs{\psi}_{C^k} = \underset{\underset{1 \leq \abs{j}\leq k}{j=(j_1,\dots,j_D)}}{\sup} \underset{x \in \mathcal{M}}{\sup} \abs{\partial_j \psi(x)}. \]
\begin{assumption}\label{assumption:H3}
 Equation~\eqref{equation:SDE} admits a spectral gap \(\lambda >0\): for every integer \(k \geq 0\) there exists a polynomial \(P_k\) such that for all \(t \geq 0\) and all \(\varphi \in C^\infty_P(\MM)\):
 \begin{equation}
 \norm{u(t,\cdot) - \inte{\mathcal{M}} \varphi(y) \rho_\infty(y) \deriv \vol_\mathcal{M}(y) }_{C^k} \leq P_k(t) e^{-\lambda t} \norm{\varphi}_{C^k}.
 \end{equation}
\end{assumption}
The Taylor expansion of the semigroup~\eqref{equation:Kolmogorov}satisfies the following.
\begin{proposition}
 Under Assumption~\ref{assumption:2.3}, for all function \(\varphi \in C^\infty_P(\MM)\) and \(h\) small enough, the following expansion holds
 \[ u(h,x) = \varphi(x) + \underset{j=1}{\overset{p}{\sum}} \frac{h^j}{j!} \mathcal{L}^j \varphi(x) + h^{p+1} R^h_p(\varphi,x), \]
 where the remainder satisfies \(\abs{R^h_p(\varphi,x)} \leq C(1+r^\kappa_w(x))\).
\end{proposition}
\begin{proof}
 A Taylor expansion yields:
 \[ u(h,x) = \varphi(x) + \underset{j=1}{\overset{p}{\sum}} \frac{h^j}{j!} \mathcal{L}^j \varphi(x) + \int_0^t \frac{{(h-t)}^p}{p!} \partial_t^{p+1} u(t,x) \deriv t. \]
 Hence, the remainder satisfies
 \begin{align*}
 \abs{R_p^h(\varphi,x)} = \abs{ \frac{h^{-1}}{p!} \int_0^t {(\frac{{(h-t)}}{h})}^p \partial_t^{p+1} u(t,x) \deriv t } \leq \frac{h^{-1}}{p!} \int_0^t C (1+{r(x)}^\kappa) \deriv t \leq \frac{C}{p!} (1+{r(x)}^\kappa),
 \end{align*}
 where we use Assumption~\ref{assumption:2.3}.
\end{proof}
Introduce the Taylor-Talay-Tubaro expansion~\cite{Talay90eot} of the numerical scheme~\eqref{equation:method} and an assumption on its moments.
\begin{assumption}\label{assumption:2.6}\label{assumption:3.2}
 For all function $\varphi \in C^\infty_P(\mathcal{M})$ and $h$ small enough, the expansion holds
 \[\espe{\varphi(X_1) | X_0 = x} = \varphi(x) + \underset{j=1}{\overset{p}{\sum}} h^j A_j \varphi(x) + h^{p+1} R^h_p(\varphi,x),\]
 where \(A_j\) is a linear differential operator of order \(2j\) and
 \(R^h_p(\varphi,x)\) is a remainder satisfying
 \[ \abs{R^h_p(\varphi,x)} \leq C(1+r^\kappa_w(x))\abs{\varphi}_{C^{2p+2}}.\]
 In the case of a post-processed method~\eqref{equation:post-process}, the expansion is of the same form,
 \[ \espe{\varphi(\overline{X}) | X} = \varphi(X) + \underset{j=1}{\overset{p}{\sum}} h^j \overline{A_j} \varphi(X) + h^{p+1} \overline{R}^h_p(\varphi,X),\]
 where \(\overline{R}^h_p(\varphi,x)\) satisfies \[ \abs{\overline{R}^h_p(\varphi,x)} \leq C(1+r^\kappa_w(x))\abs{\varphi}_{C^{2p+2}}.\]
\end{assumption}
\begin{remark}
 For most of the existing numerical methods that have a Taylor-Talay-Tubaro expansion, the differential operators \(A_j\) have a specific form. They typically write with the vector fields $F$ and $(E_d)$ and their iterated covariant derivatives, that is, they write with the coordinates of the jet space over $F$ and the frame $(E_d)$.
 In the specific case of stochastic Runge-Kutta methods, the \(A_j\) are equivariant with respect to orthogonal change of variables, which results in the \(A_j\) being naturally described by tree structures (see~\cite{Laurent20eab, Laurent23tue}). The extension of such universal equivariance property to the frozen-flow methods used here is matter for future work.
\end{remark}
\begin{definition}
 Under Assumption~\ref{assumption:2.6}, a method is consistent if \(A_1 = \mathcal{L}\). In the following, the consistency of the method is always considered true.
\end{definition}
\begin{assumption}\label{assumption:2.7}
 The method~\eqref{equation:method} has finite moments, that is, for all \(\kappa >0\) and all \(T \in \mathds{R}_+\) there exists a constant \(C(T) > 0\) such that:
 \[ \underset{n=0 \dots N}{\sup} {\espe{{r(X_n)}^\kappa | X_0 = x} } \leq C(T). \]
\end{assumption}
\subsection{High-order characterization of the invariant measure}\label{sub:Characterization of high order for the invariant measure}
Let us now state one of the central results of this work. We introduce new sufficient conditions to approximate the invariant measure, adapted from~\cite{Abdulle14hon}. This
result is key for designing the high-order numerical schemes~\eqref{equation:method} presented in Subsection~\ref{sub:High order frozen methods for sampling ergodic dynamics}. Proofs of Theorems~\ref{theoreme:characterization} and~\ref{theorem:theorem_post}, adapted from~\cite{Debussche12wbe,Vilmart15pif}, are posponed to Section~\ref{section:High order analysis for the invariant measure}.
\begin{theorem}\label{theoreme:characterization}
 Consider the SDE~\eqref{equation:SDE} on \( \mathcal{M} \) satisfying Assumptions~\ref{assumption:2.1},~\ref{assumption:H2} and~\ref{assumption:H3} and solved by an ergodic numerical method~\eqref{equation:method} satisfying Assumptions~\ref{assumption:H1} and~\ref{assumption:3.2} and
 \[ A_j^* \deriv \mu_\infty = 0 \text{ for } j=2 \dots p.\]
 Then the method~\eqref{equation:method} is of order $p$ for the invariant measure for equation~\eqref{equation:Langevin}. More precisely, the
 error of the invariant measure $e(\varphi,h)$ satisfies for all $\varphi \in
 C_P^\infty$ and all $h \in [0,h_0]$ for a small \(h_0\):
 \[ e(\varphi,h) = h^p \int_{0}^{\infty} \inte{\mathcal{M}} u(t,x) A_{p+1}^* \deriv \mu_\infty(x) \deriv t + \mathcal{O}(h^{p+1}). \]
\end{theorem}
\begin{proposition}
 Under Assumptions~\ref{assumption:2.1},~\ref{assumption:2.3},~\ref{assumption:2.6},~\ref{assumption:2.7}, and~\ref{assumption:H1}, if the Taylor expansion~\ref{assumption:3.2} satisfies the same expression that in~\cite{Bronasco25hoi}
 \[ A_j = \frac{1}{j!} \mathcal{L}^j, \quad j = 1,\dots,p,\]
 then the integrator is of order \(p\) for the invariant measure.
\end{proposition}
This characterization is extended to post-processed method~\eqref{equation:post-process}, in the spirit of~\cite{Vilmart15pif}.
\begin{theorem}\label{theorem:theorem_post}
 Under the same hypotheses as Theorem~\ref{theoreme:characterization}, consider a numerical consistent ergodic method of the form~\eqref{equation:post-process} that satisfies
 \[ A_{j+1}^* \deriv \mu_\infty = \overline{A_j}^* \deriv \mu_\infty = 0 \text{ for } j= 1 \dots p-1, \quad {\left(A_{p+1}+[\mathcal{L},\overline{A_p}]\right)}^*\deriv \mu_\infty=0.\]
 Then the method is at least of order $p+1$ for the invariant measure of equation~\eqref{equation:Langevin}.
\end{theorem}
\subsection{Post-processed frozen-flow methods of high order for sampling ergodic dynamics}\label{sub:High order frozen methods for sampling ergodic dynamics}
In this section, we leverage Theorem~\ref{theoreme:characterization} to derive
a system of conditions whose resolution yields three new order-2 methods for
accurately sampling the invariant measure. Higher orders can be obtained by
extending the Taylor expansion to the desired degree. Our analysis focuses on
frozen-flow schemes, which generalize classical Lie group
methods~\cite{Iserles00lgm,Owren99rkm}.
See Section~\ref{Section:Intrinsic order conditions with planar exotic forests} for details of the Butcher series formalism used on the Talay-Tubaro expansion.
\begin{definition}
 For all \(p,q \in \MM\), \(\exp\left(t g^d(q) E_d\right) \cdot p\) is the solution of the ODE
 \[Y'(t) = g^d(q) E_d(Y(t)), \quad Y(0) = p.\]
\end{definition}
\begin{remark}
 The Euler method \(X_{n+1} = X_n + h f^d(X_n) + \sqrt{2h} \xi^d_n\) on \(\mathds{R}^d\) becomes on \(\MM\) the frozen-flow Euler method:
 \begin{equation}\label{equation:frozenEuler}
 X_{n+1} = \exp \left( \left( h f^d(X_n) + \sqrt{2h} \xi^d_n \right) E_d \right) X_n.
 \end{equation}
 In general, the method~\eqref{equation:frozenEuler} does not coincide with the Riemannian Langevin method:
 \[ X_{n+1} = \exp^{Riem} \left( -h \nabla V (X_n) + \sqrt{2h} \xi^d_{n} E_d \right) X_n, \]
 for equation~\eqref{equation:Langevin}, based on the Levi-Civita connection and proposed in~\cite{Bharath23sae}.
\end{remark}
Following~\cite{Vilmart15pif} and the frozen-flow methods~\cite{Bronasco25hoi}, we consider the following class of methods and we define with the same way the class for the post-processor~\eqref{equation:post-process} \[\overline{X}_N = \overline{\Psi}_h(X_N).\]

\begin{align}
 H_n^i = & \exp \left( \left( h Z^0_{i,j,K} f^d(H_n^j)
 + \sqrt{h} Z_{i,K}^d \xi^{d}_n \right) E_d \right) \cdots \notag \\ & \cdots \exp \left( \left( h Z^0_{i,j,1} f^d(H_n^j)
 + \sqrt{h} Z_{i,1}^d \xi^{d}_n \right) E_d \right) X_n, \notag \\
 X_{n+1} = & \exp \left( \left( h z^0_{i,K} f^d(H_n^i)
 + \sqrt{h} z^d_{K} \xi^{d}_n \right) E_d \right) \cdots \label{equation:post-processed} \\
 & \cdots \exp \left( \left( h z^0_{i,1} f^d(H_n^i)
 + \sqrt{h} z^d_{1} \xi^{d}_n \right) E_d \right) X_n, \notag
\end{align}
where the \( \xi^{d}_n,\overline{\xi^{d}_N} \) are independent standard Gaussian random variables and where the \(Z^0_{i,j,k},z^0_{i,k},z^d_{k},Z^d_{i,k},\overline{z_{k}^{d}},\overline{z^0_{i,k}},\overline{Z_{i,k}^{d}},\overline{Z^0_{i,j,k}}\) are fixed real numbers.

\begin{remark}
    As a forest of ordre \(p\) contains at most \(2p\) trees, one only needs to consider approximations of Gaussian random variables with finite moments of all order and the same \(2p\) first moments that includes vanishing odd moments.
\end{remark}

\begin{definition}
 For a finite subset \( S \subset {\{1,\dots,K\}}^n\) of multi-indices and for \(f: S \to \mathds{R}\), the factorial sum is given by
 \[ \underset{\mathbf{k} \in S}{^!\sum} f(\mathbf{k}) = \underset{\mathbf{k} \in S}{\sum} \frac{1}{\mathbf{k}!} f(\mathbf{k}) , \quad (k_1,\ldots,k_n)! = k_1 ! \cdots k_n ! .\]
\end{definition}
We now express the general method as an exotic S-serie (see Definition~\ref{definition:Sseries}), with the coefficient map of the numerical methods \(a_w\), and the algebraic formalism of Section~\ref{Section:Intrinsic order conditions with planar exotic forests}.
\begin{proposition}
 The expression of the first-order coefficients of a method of the form~\eqref{equation:post-processed} are given by
 \[A_1 \varphi (x) = \underset{k}{{\sum}} z^0_{i,k} f^i(x) E_i[\varphi](x) + \underset{k_2 \leq k_1}{{^!\sum}} z_{k_1}^{d_1} z_{k_2}^{d_1} E_{d_1} [ E_{d_1}[\varphi]](x) .\]
 For the second order, the expression of coefficients of \(A_2\) are given in Table~\ref{table:tableau_Coefficient_general}.
 \begin{table}[H]
 \centering
 \begin{tabular}{c|c|c|c} 
 {Forest \(\pi\)} & {Differential op.\ \(\mathds{F}^F(\pi)\)} & {Coeff.\ \(a_w(\pi)\) of \(A_2\) } & { Coeff.\ \(\crochet{\mathcal{L}}{\overline{A_1}}\) } \\
\hline 
 \(\forestA\) & \( f^j E^j[f^i] E_i[\varphi] \) & $ \underset{k}{{\sum}} z^0_{i,k} Z^0_{i,j,k} $ & $0$\\
 \(\forestB\) & \( E_{d_1} [ E_{d_1}[f^i]]E_i[\varphi] \) & $ \underset{k}{{\sum}} \underset{k_2 \leq k_1}{{^!\sum}} z^0_{i,k} Z_{i,k_2}^{d_1} Z_{i,k_1}^{d_1} $ & $ C_{\crochet{\mathcal{L}}{\overline{A_1}}} $ \\
 \(\forestC\) & \(f^j f^i E_j[E_i[\varphi]]\) & $ \underset{k_2 \leq k_1}{{^!\sum}} z^0_{j,k_2} z^0_{i,k_1} $ & $0$\\
 \(\forestD\) & \( E_{d_1}[f^i] E_i[E_{d_1}[\varphi]] \) & $ \underset{k}{{\sum}} \underset{k_2 \leq k_1}{{^!\sum}} (z_{i,k}^{d_1} z^0_{i,k_2}) z_{k_1}^{d_1} $ & $0$\\
 \(\forestE\) & \(E_{d_1}[f^i] E_{d_1}[E_i[\varphi]]\) & $ \underset{k}{{\sum}} \underset{k_2 \leq k_1}{{^!\sum}} z_{k_2}^{d_1} (z_{i,k}^{d_1} z^0_{i,k_1}) $ & $ 2 C_{\crochet{\mathcal{L}}{\overline{A_1}}} $\\
 \(\forestF\) & \(f^i E_i[E_{d_1}[E_{d_1}[\varphi]]]\) & $ \underset{k_3 \leq k_2 \leq k_1}{{^!\sum}} z^0_{i,k_3} z_{k_2}^{d_1} z_{k_1}^{d_1} $ & $-C_{\crochet{\mathcal{L}}{\overline{A_1}}} $\\
 \(\forestG\) & \(f^i E_{d_1}[E_i[E_{d_1}[\varphi]]]\) & $ \underset{k_3 \leq k_2 \leq k_1}{{^!\sum}} z_{k_3}^{d_1} z^0_{i,k_2} z_{k_1}^{d_1} $ & $0$\\
 \(\forestH\) & \(f^i E_{d_1}[E_{d_1}[E_i[\varphi]]]\) & $ \underset{k_3 \leq k_2 \leq k_1}{{^!\sum}} z_{k_3}^{d_1} z_{k_2}^{d_1} z^0_{i,k_1} $ & $ C_{\crochet{\mathcal{L}}{\overline{A_1}}} $\\
 \(\forestI\) & \(E_{d_2}[E_{d_2}[E_{d_1}[E_{d_1}[\varphi]]]]\) & $\underset{k_4 \leq k_3 \leq k_2 \leq k_1}{{^!\sum}} z_{k_4}^{d_2}z_{k_3}^{d_2}z_{k_2}^{d_1}z_{k_1}^{d_1}$ & $0$\\
 \(\forestJ\) & \(E_{d_2}[E_{d_1}[E_{d_2}[E_{d_1}[\varphi]]]]\) & $\underset{k_4 \leq k_3 \leq k_2 \leq k_1}{{^!\sum}} z_{k_4}^{d_2}z_{k_3}^{d_1}z_{k_2}^{d_2}z_{k_1}^{d_1}$ & $0$\\
 \(\forestK\) & \(E_{d_1}[E_{d_2}[E_{d_2}[E_{d_1}[\varphi]]]]\) & $\underset{k_4 \leq k_3 \leq k_2 \leq k_1}{{^!\sum}} z_{k_4}^{d_1}z_{k_3}^{d_2}z_{k_2}^{d_2}z_{k_1}^{d_1}$ & $0$\\
 \end{tabular}
 \caption{\scriptsize S-series coefficients of \(A_2 \varphi = S_h(a_w) \vartriangleright \varphi\) for a method of the form~\eqref{equation:post-processed}, where the sums over the \(d_q\) and the \(i,j\) are omitted for conciness and where we write \(C_{\crochet{\mathcal{L}}{\overline{A_1}}} = \underset{k}{{\sum}} \overline{z}^0_{i,k} - \underset{k_2 \leq k_1}{{^!\sum}} \overline{z}_{k_1}^{d_1} \overline{z}_{k_2}^{d_1}\).}\label{table:tableau_Coefficient_general}
 \end{table}
\end{proposition}
\begin{remark}
 The use of a post-processor only modifies the order conditions by the coefficient
 \( C_{\crochet{\mathcal{L}}{\overline{A_1}}} = \underset{k}{{\sum}} \overline{z}^0_{i,k} - \underset{k_2 \leq k_1}{{^!\sum}} \overline{z}_{k_1}^{d_1} \overline{z}_{k_2}^{d_1}. \)
 If \(\overline{A_1} \propto \mathcal{L} \), we thus recover the order conditions without post-processor.
\end{remark}
Since we only express our conditions in terms of exotic forests, to simplify the notation, we add the coefficients of \(\crochet{\mathcal{L}}{\overline{A_1}}\), in the post-processed case, in thoses of \(A_2\).

Using the Butcher series formalism of Section~\ref{Section:Intrinsic order conditions with planar exotic forests}, we obtain a set of conditions that ensures the order 2 for the invariant measure of equation~\eqref{equation:Langevin}.
\begin{theorem}\label{theorem:conditions2}
 Under the assumptions of Theorems~\ref{theoreme:characterization} and~\ref{theorem:theorem_post}, a method~\eqref{equation:method} or a post-processed method~\eqref{equation:post-process} is of order 2 for the invariant measure for equation~\eqref{equation:Langevin} if the following 5 order conditions are satisfied:
 \[ a_w({\forestA}) = a_w({\forestB}) , \quad a_w({\forestF}) = a_w({\forestI}) , \quad a_w({\forestD}) + a_w({\forestJ}) = a_w({\forestG}) ,\]
 \[ a_w({\forestA}) + a_w({\forestH}) = a_w({\forestE}) + a_w({\forestK}) , \quad a_w({\forestC}) + a_w({\forestK}) = a_w({\forestD}) + a_w({\forestH}) . \]
 \begin{proof}
 Table~\ref{table:order2} gives the reduction of exotic forests of order 2. Therefore, the S-series
 \[A_2 \varphi = S_h(a_w) \vartriangleright \varphi\]
 can be expressed as a linear combination of irreducible forests, whose coefficients provide the second-order conditions. The details are postponed to Subsection~\ref{sub:Integration by parts for the invariant measure}. 
 \end{proof}
\end{theorem}
\begin{remark}
 Recall from~\cite{Bronasco22ebs}, for \(\MM = \mathds{R}^D\), the forests are not ordered and there are one condition for the first order and three for the second one.
\end{remark}
\begin{remark}
 For the second order, there are 5 equations with 11 variables and for the third order, there are 40 equations and with 95 variables.
 We recall from~\cite{Bronasco25hoi} that for the weak error, there are 8 equations for the order 2 and 73 for the order 3, so the invariant measure allows for a significant reduction of the number of order conditions.
\end{remark}
We propose the following new post-processed frozen-flow method to solve equation~\eqref{equation:Langevin}. In the Euclidean setting, Method~\ref{method:pp} rewrites as the Leimkuhler-Matthews method~\cite{Leimkuhler13rco} with the post-processor formulation from~\cite{Vilmart15pif}.
\begin{algorithm}[H]
 \caption{Post-processed frozen-flow method for equation~\eqref{equation:Langevin}}
 \begin{algorithmic}\label{method:pp}
 \STATE$H_n = \exp \left( \frac{\sqrt{2h}}{2} \xi^d_n \right) E_d X_n$,
 \STATE$X_{n+1} = \exp \left( \frac{5h}{4} F(H_n) + \frac{\sqrt{2h}}{4} \xi^d_n E_d \right)\exp \left( - \frac{h}{4} F(H_n) + \frac{3\sqrt{2h}}{4} \xi^d_n E_d \right) X_n$,
 \STATE$\overline{X_N} = \exp \left( \frac{\sqrt{2h}}{2} \overline{\xi}^d_N E_d \right) X_N$.
 \end{algorithmic}
 \end{algorithm}
\begin{corollary}
 Method~\ref{method:pp} is explicit, of order 2 for the invariant measure of equation~\eqref{equation:Langevin} and weak order one, and uses one evaluation of $F$ and three exponential maps per step.
\end{corollary}
We propose two alternative schemes without post-processors to solve equation~\eqref{equation:Langevin}, for the sake of comparison.
 \begin{algorithm}[H]
 \caption{Stochastic frozen-flow Heun method for equation~\eqref{equation:Langevin}}
 \begin{algorithmic}\label{methode1}
 \STATE$H_n = \exp \left( h F(X_n) + \sqrt{2h} \xi^d_n E_d \right) X_n$,
 \STATE$X_{n+1} = \exp \left( \frac{h}{2} F(H_n) \right) \exp \left( \frac{h}{2} F(X_n) + \sqrt{2h} \xi^d_n E_d \right) X_n $.
 \end{algorithmic}
 \end{algorithm}

 \begin{algorithm}[H]
 \caption{2-steps frozen-flow Runge-Kutta method for equation~\eqref{equation:Langevin}}
 \begin{algorithmic}\label{methode2}
 \STATE$H_n = \exp \left( \frac{h}{4} F(X_n) + \frac{\sqrt{2h}}{2} \xi^d_n E_d \right) X_n$,
 \STATE$X_{n+1} = \exp \left( \frac{h}{6} F(X_n) + \frac{2h}{3} F(H_n) + \frac{\sqrt{2h}}{2} \xi^d_n E_d \right) \exp \left( \frac{h}{6} F(X_n) + \frac{\sqrt{2h}}{2} \xi^d_n E_d \right) X_n$.
 \end{algorithmic}
 \end{algorithm}

\begin{corollary}
 Methods~\ref{methode1} and~\ref{methode2} are of first order for the weak error and of second order for the invariant measure of equation~\eqref{equation:Langevin}.
\end{corollary}

\section{High order analysis for the invariant measure}\label{section:High order analysis for the invariant measure}
In this subsection, we generalize the analysis of~\cite{Debussche12wbe} to Riemannian manifolds and prove Theorem~\ref{theoreme:characterization} and Theorem~\ref{theorem:theorem_post}.

Let \( {(P_t)}_{t\geq0} \) be the semigroup associated to the Markov
process \({(X(t,x))}_{t\geq0,x\in\mathcal{M}}\), in other words,
\({(P_t)}_{t\geq0} \) gives to \(\psi\) a smooth function on \(\mathcal{M}\),
the unique solution \(u\) of the Kolmogorov equation~\eqref{equation:Kolmogorov}.
In the spirit of backard analysis~\cite{Debussche12wbe,Hairer06gni}, let us build a modified generator \(\tilde{\mathcal{L}_\tau} \) such that the solution of the problem
\[ \partial_t v(t,x) = \tilde{\mathcal{L}_\tau} v , \quad v(0,x) = \varphi(x), \]
at the time \(t = \tau\), coincides with the numerical flow, that is
\[ \exp \left( \tau \tilde{\mathcal{L}_\tau} \right) = \id + \underset{n \geq 1}{\sum} \tau^n A_n.\]
Thus the expansion yields
\[ \tilde{\mathcal{L}_\tau} = \underset{k = 0}{\overset{n}{\sum}} \tau^k L_k + R_n , \quad R_n = \underset{k \geq n+1}{{\sum}} \tau^k L_k . \]
Define \(\pi_n : P \mapsto n! {\left({\partial_\tau}^n P\right)}_{| \tau = 0} \), then
\begin{align*}
 A_n & = \pi_n \left( \exp \left( \tau \tilde{\mathcal{L}_\tau} \right) \right) = \underset{k \geq 0}{\sum} \frac{1}{k!} \pi_n \left( \tau^k \tilde{\mathcal{L}_\tau}^k \right) \\
 & = \underset{k = 1}{\overset{n}{\sum}} \frac{1}{k!} \pi_{n-k} \left( \tilde{\mathcal{L}_\tau}^k \right) = \underset{k = 1}{\overset{n}{\sum}} \frac{1}{k!} \pi_{n-k} \left( { \left( L_0 + \tau L_1 + \cdots + \tau^n L_n + R_n \right) }^k \right) \\
 & = \underset{k = 1}{\overset{n}{\sum}} \frac{1}{k!} \pi_{n-k} \left( \underset{j = 0}{\overset{nk}{\sum}} \tau^j \underset{i_1+ \cdots +i_k = j}{\sum} L_{i_1} \cdots L_{i_k} \right) \\
 & = \underset{k = 1}{\overset{n}{\sum}} \frac{1}{k!} \underset{i_1+ \cdots +i_k = n-k}{\sum} L_{i_1} \cdots L_{i_k} , 
\end{align*}
where we use \( \pi_p (R_n) = 0 \) when \(p \leq n\).
Thus \(\suite{L_n}{n}\) satisfies the recurrence relation
\[ L_0 = A_1 = \mathcal{L} , \quad L_n = A_{n+1} - \underset{k = 2}{\overset{n+1}{\sum}} \frac{1}{k!} \underset{i_1+ \cdots +i_k = n+1-k}{\sum} L_{i_1} \cdots L_{i_k}. \]
\begin{proposition}\label{proposition:4.1}
 Let \( \varphi \) be in \(C_P^\infty(\mathcal{M})\). Under Assumptions~\ref{assumption:H3} and~\ref{assumption:3.2}, there exists a sequence \( \suite{v_l}{l} \), such that for all \(n \geq 0\):
 \[ \partial_t v_n(t,x) - \mathcal{L} v_n(t,x) = \underset{l=1}{\overset{n}{\sum}} L_l v_{n-l}(t,x), \quad v_0(0,x) = \varphi(x), \quad v_n(0,x) \equiv 0 \text{ for } n > 0. \]
 Let \(v^N(t,x) =
 \underset{n=0}{\overset{N}{\sum}} \tau^n v_n(t,x) \), the we have
 \[ \norm{\espe{v^N(t,X_1)} -v^N(t+\tau,x) }_\infty \leq C_N \tau^{N+1} \underset{n=0,\dots N}{\underset{s \in (0,\tau)}{\sup}}\norm{v_n(t_j+s,x)}_{C^{4N+2}}. \]
\end{proposition}
\begin{proof}
 For \(N=0\), set \(v_0 :(t,x) \mapsto u(t,x) \).
 By induction, assume $v_0,\dots,v_N$ are defined and denote \(F_{N+1}(t,x) = \underset{l=1}{\overset{n+1}{\sum}} L_l v_{n+1-l}(t,x) \). Then the Duhamel formula ensures that
 \[ v_{N+1}(t,x) = \int_0^t P_{t-s} F_{N+1}(s,x)\deriv s. \]
 Assumption~\ref{assumption:H3} implies for all \( n \geq 0 \), \( v_n \) is
 smooth and
 \begin{equation}\label{equation:4.13} \norm{v_n(t)}_{C^k} \leq C \norm{\varphi}_{C^{k+4n}}. \end{equation}
 The case \(n=0\) is ensured by Assumption~\ref{assumption:H3}. Suppose that equation~\eqref{equation:4.13} is true for $0,\dots,n$. As \(L_l\) is of order \(2+2\), then for all \(l=1,\dots,n+1\),
 \[ \norm{L_l v_{n+1-l}(t)}_{C^k} \leq C \norm{v_{n+1-l}(t)}_{C^{k+2l+2}}\]
 and by induction \(\norm{v_{n+1-l}(t)}_{C^{k+2l+2}} \leq C
 \norm{\varphi}_{C^{(k+2l+2)+4(n+1-l)}}\), that is
 \[ \norm{F_{n+1}(t)}_{C^k} \leq C \norm{\varphi}_{C^{k+4(n+1)}}.\]
 Thus, from Assumption~\ref{assumption:H3}, we obtain
 \[\norm{P_{t-s} F_{n+1}(s)}_{C^{k}} \leq P_k(t-s) e^{-\lambda (t-s) } \norm{F_{n+1}(s)}_{C^{k}},\]
 and
 \[ \norm{v_{n+1}(t)}_{C^k} \leq C \norm{\varphi}_{C^{k+4(n+1)}}.\]
 Hence the proof by induction.

 Fix \(t \geq 0\) and set \(w_n : (s,x) \mapsto v_n(t+s,x)\). Then \(w_n\)
 satisfies
 \[ \partial_s w_n(s,x) - \mathcal{L} w_n(s,x) = \underset{l=1}{\overset{n}{\sum}} L_l w_{n-l}(s,x), \quad w_n(0,x) = v_n(t,x). \]
 Hence \(\partial_s w_n(s,x) = \underset{l=0}{\overset{n}{\sum}} L_l
 w_{n-l}(s,x)\),
 \[\partial_s^2 w_n(s,x) = \underset{l=0}{\overset{n}{\sum}} L_l \underset{k=0}{\overset{n-l}{\sum}} L_k w_{n-l-k}(s,x) = \underset{k=0}{\overset{n}{\sum}} \underset{l_1+l_2=k}{\sum} L_{l_1} L_{l_2} w_{n-k}(s,x), \]
 and by induction, it follows
 \[\partial_s^m w_n(s,x) = \underset{l_1+ \cdots +l_{m+1}=n}{\sum} L_{l_1} \dots L_{l_m} w_{l_{m+1}}(s,x). \]
 Noticing that \(L_l\) are differentials operators of order $2l+2>0$, the inegality holds
 \[ \norm{\partial_s^m w_n(s,x)}_\infty \leq C \underset{k=0,\dots,n}{\sup} \abs{w_k(s)}_{C^{2n+2m}}. \]
 A Taylor expansion yields, for the order $2N-n$,
 \begin{align*}
 w_n(\tau,x) = & \underset{m=0}{\overset{2N-n}{\sum}} \frac{\tau^m}{m!} \partial_s^m w_n(0,x) + \int_0^\tau \frac{{(\tau-s)}^{2N-n}}{(2N-n)!} \partial_s^{2N-n+1} w_n(s,x) \deriv s \\
 = & \underset{m=0}{\overset{2N-n}{\sum}} \frac{\tau^m}{m!} \underset{l_1+ \cdots +l_{m+1}=n}{\sum} L_{l_1} \dots L_{l_m} w_{l_{m+1}}(0,x) + R_{2N,n}(\tau,x),
 \end{align*}
 where \[ \norm{R_{2N,n}(\tau)}_\infty \leq C \tau^{2N-n+1} \underset{k=0,\dots 2N}{\underset{s \in (0,\tau)}{\sup}}\norm{w_k(s,x)}_{C^{4N+2}} \leq C \tau^{N+1}. \]
 On one hand, we have
 \begin{align*}
 v^N(t+\tau,x) = & \underset{n=0}{\overset{N}{\sum}} \underset{m=0}{\overset{2N-n}{\sum}} \frac{\tau^{m+n}}{m!} \underset{l_1+ \cdots +l_{m+1}=n}{\sum} L_{l_1} \dots L_{l_m} w_{l_{m+1}}(0,x) + R_{2N}(\tau,x) \\
 = & \underset{p=0}{\overset{2N}{\sum}} \tau^p \underset{m=0}{\overset{p}{\sum}} \frac{1}{m!} \underset{l_1+ \cdots +l_{m}=p-m-l_{m+1}}{\sum} L_{l_1} \dots L_{l_m} w_{l_{m+1}}(0,x) + R_{2N}(\tau,x) \\
 = & \underset{p=0}{\overset{2N}{\sum}} \tau^p \underset{q=0}{\overset{p}{\sum}} \left( \underset{m=0}{\overset{p-q}{\sum}} \frac{1}{m!} \underset{l_1+ \cdots +l_{m}=p-q-m}{\sum} L_{l_1} \dots L_{l_m} \right) w_{q}(0,x) + R_{2N}(\tau,x) \\
 = & \underset{p=0}{\overset{2N}{\sum}} \tau^p \underset{q=0}{\overset{p}{\sum}} A_{p-q} w_{q}(0,x) + R_{2N}(\tau,x).
 \end{align*}
 On the other hand, we have
 \begin{align*}
 \underset{n=0}{\overset{N}{\sum}} \tau^n A_n v^N(t+\tau,x) = & \underset{n=0}{\overset{N}{\sum}} \tau^n A_n \left( \underset{q=0}{\overset{N}{\sum}} \tau^q v_q(t+\tau,x) \right) \\
 = & \underset{n=0}{\overset{N}{\sum}} \underset{q=0}{\overset{N}{\sum}} \tau^{n+q} A_n v_q(t+\tau,x) \\
 = & \underset{q=0}{\overset{N}{\sum}} \underset{p=0}{\overset{N+q}{\sum}} \tau^{p} A_{p-q} v_q(t+\tau,x) \\
 = & \underset{p=0}{\overset{2N}{\sum}} \tau^p \underset{p-q\leq N}{\underset{q=0}{\overset{p}{\sum}}} A_{p-q} w_{q}(0,x) + R_{2N}(\tau,x).
 \end{align*}
 Thus, the modified flow satisfies
 \begin{align*}
 v^N(t+\tau,x) - \underset{n=0}{\overset{N}{\sum}} \tau^n A_n v^N(t+\tau,x) = & \underset{p=0}{\overset{2N}{\sum}} \tau^p \underset{p-q > N}{\underset{q=0}{\overset{p}{\sum}}} A_{p-q} w_{q}(0,x) + R_{2N}(\tau,x) \\
 = & \underset{p=N+1}{\overset{2N}{\sum}} \tau^p \underset{q=0}{\overset{p-(N+1)}{\sum}} A_{p-q} w_{q}(0,x) + R_{2N}(\tau,x),
 \end{align*}
 and
 \begin{align*}
 \abs{v^N(t+\tau,x) - \underset{n=0}{\overset{N}{\sum}} \tau^n A_n v^N(t+\tau,x)} & \leq C \tau^{N+1} \underset{n=0,\dots N}{\sup}\norm{v_n(t,x)}_{C^{4N}} \\
 & + C \tau^{N+1} \underset{k=0,\dots 2N}{\underset{s \in (0,\tau)}{\sup}}\norm{w_k(s,x)}_{C^{4N+2}} .
 \end{align*}
 Applying Assumption~\ref{assumption:2.6} with \(\varphi = v^N(t)\) gives
 \begin{align*}
 \norm{\espe{v^N(t,X_1)} -v^N(t+\tau,x) }_\infty \leq & \norm{\espe{v^N(t,X_1)} - \underset{n=0}{\overset{N}{\sum}} \tau^n A_n v^N(t+\tau,x) }_\infty \\
 + & \norm{\underset{n=0}{\overset{N}{\sum}} \tau^n A_n v^N(t+\tau,x) - v^N(t+\tau,x) }_\infty \\
 \leq & C_N \tau^{N+1} \underset{n=0,\dots N}{\underset{s \in (0,\tau)}{\sup}}\norm{v_n(t_j+s,x)}_{C^{4N+2}}.
 \end{align*}
 Hence the result.
\end{proof}
We construct a sequence of measure which weakly converges to the invariant
measure of the numerical method~\eqref{equation:method}.
\begin{lemma}\label{lemma:sequence}
 Under Assumptions~\ref{assumption:H1},~\ref{assumption:H2},~\ref{assumption:H3} and~\ref{assumption:2.6}, let \(\suite{v_n}{n}\) and \(\suite{v^{N}}{N}\) be defined as in Proposition~\ref{proposition:4.1}. There exists $\suite{\rho_n}{n}$ such that $\rho_0 = \rho_\infty$ and for all \(N \geq 1 \):
 \[ \inte{\mathcal{M}} \rho_N(y) \deriv \vol_\mathcal{M}(y) = 0, \quad \mathcal{L}^* \rho_{N} \deriv \vol_\mathcal{M} = \underset{n=1}{\overset{N}{\sum}} L_n^* \rho_{N-n} \deriv \vol_\mathcal{M}.\]
\end{lemma}
\begin{proof}
 Let us show by induction the existence of the sequence $\suite{\rho_n}{n}$. For $n=0$, set $\rho_0 = \rho_\infty$. Suppose having constructed $\rho_0,\dots,\rho_N$ that satisfy these conditions. Then, set $G_{N+1}= \underset{n=1}{\overset{N}{\sum}} L_n^* \rho_{N-n} \deriv \vol_\mathcal{M}$. The function $G_{N+1}$ is smooth and its integral is
 \[ \inte{\mathcal{M}} G_{N+1}(x) \deriv \vol_\mathcal{M}(x) = - \underset{n=1}{\overset{N}{\sum}} \inte{\mathcal{M}} \rho_{N-n}(x) L_n \mathds{1}(x) \deriv \vol_\mathcal{M}(x)=0,\]
 as \( L_n \) is a differential operator of order $2n+2$. By
 Assumption~\ref{assumption:H2}, there exists $\rho_{N+1} \in
 C_P^\infty(\mathcal{M})$ such that $ \mathcal{L}^* \rho_{N+1} \deriv \vol_\mathcal{M} = G_{N+1} $ and $
 \inte{\mathcal{M}} \rho_{N+1}(y) \deriv \mu_\infty (y) =0$, which
 concludes the induction.
\end{proof}
\begin{proposition}\label{proposition:5.2}
 Under Assumptions~\ref{assumption:H1},~\ref{assumption:H2},~\ref{assumption:H3}, and~\ref{assumption:2.6}, let \(\suite{v_n}{n}\) and \(\suite{F_n}{n}\) be defined as in Proposition~\ref{proposition:4.1}.
 For all \(n,k \geq 0\), there exists a polynomial \(P_{k,n}\) such that
 \[ \norm{v_n(t,x) - \inte{\mathcal{M}} \varphi(y) \rho_n(y) \deriv \vol_\mathcal{M}(y)} \leq P_{k,n}(t)e^{-\lambda t} \norm{\varphi - \moy{\varphi}}_{C^{k+4n}}. \]
\end{proposition}
\begin{proof}
 For \(n=0\), \(v_0=u\) and \(\rho_0=\rho_\infty\) and by Assumption~\ref{assumption:H3}, for all integer \(k\), we have
 \[ \norm{u-\moy{\varphi}}_{C^k} \leq P_k(t)e^{-\lambda t} \norm{\varphi - \moy{\varphi}}_{C^k}. \]
 Suppose that the induction is true for \( 0 ,\dots, n \). Setting
 \[ {c_{n+1}(t) = \underset{m=0}{\overset{n+1}{\sum}} \inte{\mathcal{M}} v_{n+1-m}(t,x) \rho_m(x) \deriv \vol_\mathcal{M}(x)} , \quad \]
 then \(c_{n+1}\) is derivable and its derivative is
 \begin{align*}
 c_{n+1}'(t) = & \underset{m=0}{\overset{n+1}{\sum}} \inte{\mathcal{M}} \partial_t v_{n+1-m}(t,x) \rho_m(x) \deriv \vol_\mathcal{M}(x) \\
 = & \underset{m=0}{\overset{n+1}{\sum}} \inte{\mathcal{M}} \partial_t v_{m}(t,x) \rho_{n+1-m}(x) \deriv \vol_\mathcal{M}(x) \\
 = & \underset{m=0}{\overset{n+1}{\sum}} \inte{\mathcal{M}} \underset{l=0}{\overset{m}{\sum}} L_{m-l} v_l(t,x) \rho_{n+1-m}(x) \deriv \vol_\mathcal{M}(x) \\
 = & \underset{m=0}{\overset{n+1}{\sum}} \underset{l=0}{\overset{m}{\sum}} \inte{\mathcal{M}} v_l(t,x) L_{m-l}^* \rho_{n+1-m}(x) \deriv \vol_\mathcal{M}(x) \\
 = & \underset{l=0}{\overset{n+1}{\sum}} \inte{\mathcal{M}} v_l(t,x) \left( \underset{m=l}{\overset{n+1}{\sum}} L_{m-l}^* \rho_{n+1-m}(x) \right) \deriv \vol_\mathcal{M}(x) \\
 = & \underset{l=0}{\overset{n+1}{\sum}} \inte{\mathcal{M}} v_l(t,x) \left( \underset{r=0}{\overset{n+1-l}{\sum}} L_{r}^* \rho_{n+1-r-l}(x) \right) \deriv \vol_\mathcal{M}(x).
 \end{align*}
 Although, by definition of \( \rho_{n} \), we have \( \underset{r=0}{\overset{n+1-l}{\sum}} L_{r}^* \rho_{n+1-r-l} \deriv \vol_\mathcal{M} = 0 \), so \(c_{n+1}\) is constant and
 \[ \inte{\mathcal{M}} \partial_t v_{n+1}(t,x) \deriv \mu_\infty (x) = - \underset{m=1}{\overset{n+1}{\sum}} \inte{\mathcal{M}} \partial_t v_{n+1-m}(t,x) \rho_{m}(x) \deriv \vol_\mathcal{M}(x). \]
 In addition, the computation of the mean \(F_{n+1}\) with respect to the
 measure \(\rho_\infty\) ensures
 \begin{align*}
 \moy{F_{n+1}(t)} = & \inte{\mathcal{M}} F_{n+1}(t,x) \deriv \mu_\infty (x) \\
 = & \inte{\mathcal{M}} \left(\partial_t v_{n+1}(t,x) - \mathcal{L} v_{n+1}(t,x)\right) \deriv \mu_\infty (x) \\
 = & - \underset{m=1}{\overset{n+1}{\sum}} \inte{\mathcal{M}} \partial_t v_{n+1-m}(t,x) \rho_{m}(x) \deriv \vol_\mathcal{M}(x),
 \end{align*}
 by Proposition~\ref{proposition:4.1} and \(\mathcal{L}^* \deriv \mu_\infty = 0\). By computing the expression of \(v_{n+1}\), we find
 \begin{align*}
 v_{n+1}(t,x) = & \int_0^t \moy{F_{n+1}(s)} \deriv s + \int_0^t P_{t-s} \left( F_{n+1}(s,x) \moy{F_{n+1}(s)} \right) \deriv s \\
 = & - \underset{m=1}{\overset{n+1}{\sum}} \inte{\mathcal{M}} v_{n+1-m}(t,x) \rho_{m}(x) \deriv \vol_\mathcal{M}(x) + \inte{\mathcal{M}} \varphi(x) \rho_{n+1}(x)\deriv \vol_\mathcal{M}(x) \\
 + & \int_0^t P_{t-s} \left( F_{n+1}(s,x) \moy{F_{n+1}(s)} \right) \deriv s.
 \end{align*}
 We have
 \begin{align*}
 \underset{m=1}{\overset{n+1}{\sum}} & \inte{\mathcal{M}} v_{n+1-m}(t,x) \rho_{m}(x) \deriv \vol_\mathcal{M}(x) = \\
 & \underset{m=1}{\overset{n+1}{\sum}} \inte{\mathcal{M}} \left( v_{n+1-m}(t,x) - \inte{\mathcal{M}} \varphi(y) \rho_{n+1-m}(y) \deriv \vol_\mathcal{M}(y) \right) \rho_{m}(x) \deriv \vol_\mathcal{M}(x),
 \end{align*}
 as the mean of \(\rho_m\) is zero.
 Moreover, the inequality holds
 \begin{align*}
 \norm{F_{n+1}(s) - \moy{F_{n+1}(s)}}_{C^k} \leq & \underset{l=1}{\overset{n+1}{\sum}} \norm{L_l v_{n+1}(s) - \moy{L_l v_{n+1}(s)}}_{C^k} \\
 \leq & C_k \underset{l=1}{\overset{n+1}{\sum}} \norm{ v_{n+1-l}(s) - \moy{ v_{n+1-l}(s)}}_{C^{k+2l+2}} \\
 \leq & C_k \underset{l=0}{\overset{n}{\sum}} \norm{ v_{l}(s) - \moy{ v_{l}(s)}}_{C^{k+2(n+1-l)+2}},
 \end{align*}
 and, the proof is concluded by
 \begin{align*}
 \Vert v_{n+1}(t,x) & - \inte{\mathcal{M}} \varphi(x) \rho_{n+1}(x)\deriv \vol_\mathcal{M}(x) \Vert_{C^k} \\
 \leq & \underset{m=1}{\overset{n+1}{\sum}} \norm{v_{n+1-m}(t,x) - \inte{\mathcal{M}} \varphi(y) \rho_{n+1-m}(y) \deriv \vol_\mathcal{M}(y)}_{C^k} \inte{\mathcal{M}} \abs{\rho_{m}(x)} \deriv \vol_\mathcal{M}(x) \\
 + & \int_0^t p_k(t-s) e^{-\lambda (t-s)} \norm{F_{n+1}(s) - \moy{F_{n+1}(s)}}_{C^k} \deriv s \\
 \leq & P_{k,n+1}(t) e_{-\lambda t} \norm{\varphi - \moy{\varphi}}_{C^{4+4(n+1)}},
 \end{align*}
 by induction and Assumption~\ref{assumption:H3}.
\end{proof}
To be able to prove Theorem~\ref{theoreme:characterization}, we recall
a lemma from~\cite{Abdulle14hon}, that is the central result for the
proof of the theorem.
\begin{lemma}\label{lemma:second}
 Under the same assumptions of Lemma~\ref{lemma:sequence} and Proposition~\ref{proposition:5.2}, setting $ \rho_N^h = \underset{n=0}{\overset{N}{\sum}} h^n \rho_{n} $, there exists $C_N >0$ such that
 \[ \abs{ \inte{\mathcal{M}} \varphi(y) \rho^h(y) \deriv \vol_\mathcal{M}(y) - \inte{\mathcal{M}} \varphi(y) \rho_N^h(y) \deriv \vol_\mathcal{M}(y) } \leq C_N \norm{\varphi}_{C^{8N+10}} h^{N+1}.\]
 \begin{proof}
 We bound the difference between \(\espe{\varphi(X_p)}\) and the mean of \(\varphi\) with respect to the measure \(\rho_N^h\),
 \begin{align*}
 \espe{\varphi(X_p)} - & v^{N+1}(t_p,x) = \espe{v^{N+1}(0,X_p)} - v^{N+1}(t_p,x) \\
 = & \espe{\underset{j=0}{\overset{p-1}{\sum}} \espe{ v^{N+1}(t_j,X_{p-j}) - v^{N+1}(t_{j+1},X_{p-(j+1)}) \vert X_{p-(j+1)}} } \\
 = & \espe{\underset{j=0}{\overset{p-1}{\sum}} \espe{ v^{N+1}(t_j,X_1(X_{p-(j+1)})) - v^{N+1}(t_{j+1},X_{p-(j+1)}) \vert X_{p-(j+1)}} }.
 \end{align*}
 However, we recall Proposition~\ref{proposition:4.1},
 \[ \begin{array}{l}
 \espe{v^{N+1}(t_j,X_1(X_{p-(j+1)})) - v^{N+1}(t_{j+1},X_{p-(j+1)}) \vert X_{p-(j+1)}} \\
 \leq C_N \tau^{N+2} \underset{n=0,\dots N}{\underset{s \in (0,\tau)}{\sup}}\norm{v_n(t_j+s,x)}_{C^{4N+6}}.
 \end{array}\]
 It follows that
 \begin{align*}
 \espe{\varphi(X_p)} - v^{N+1}(t_p,x) \leq & C_N \tau^{N+2} \underset{j=0}{\overset{p-1}{\sum}} \underset{n=0,\dots N}{\underset{s \in (0,\tau)}{\sup}}\norm{v_n(t_j+s,x)}_{C^{4N+6}} \\
 \leq & C_N \tau^{N+2} \underset{j=0}{\overset{p-1}{\sum}} Q_N(t_j) e^{-\lambda t_j} \norm{\varphi}_{C^{8N+10}} \\
 \leq & C_N \tau^{N+1} \norm{\varphi}_{C^{8N+10}}.
 \end{align*}
 Since \( v^{N+1}(t_p,x) = v^{N}(t_p,x) + \tau^{N+1} v_{N+1}(t_p,x) \) and \( \norm{v_{N+1}(t_p,x)}_{C^0} \leq C_N \norm{\varphi}_{C^{4(N+1)}} \), we find
 \[ \abs{ \espe{\varphi(X_p)} - v^{N}(t_p,x) } \leq C_N \tau^{N+1} \norm{\varphi}_{C^{8N+10}}. \]
 Hence
 \begin{align*}
 & \abs{ \espe{\varphi(X_p)} - \inte{\mathcal{M}} \varphi(y) \rho_N^h(y) \deriv \vol_\mathcal{M}(y) } \\
 & \leq \abs{ \espe{\varphi(X_p)} - v^{N}(t_p,x) } + \abs{ v^{N}(t_p,x) - \inte{\mathcal{M}} \varphi(y) \rho_N^h(y) \deriv \vol_\mathcal{M}(y) } \\
 & \leq C_N \tau^{N+1} \norm{\varphi}_{C^{8N+10}} + \underset{n=0}{\overset{N}{\sum}} P_{0,N}(t_p) e^{-\lambda t_p} \norm{\varphi}_{C^{0+4n}} \\
 & \leq \left( C_N \tau^{N+1} + P_N(t_p) e^{-\lambda t_p} \right) \norm{\varphi}_{C^{8N+10}}.
 \end{align*}
 Hence the result by taking the limit in $p$.
 \end{proof}
\end{lemma}
With this analysis, we now prove the
Theorem~\ref{theoreme:characterization}.
\begin{proof}[Proof of Theorem~\ref{theoreme:characterization}]
 Under Assumptions~\ref{assumption:H2}, we set the sequence \(\suite{\rho_n}{n}\) defined in Lemma~\ref{lemma:sequence}.
 The error is defined by
 \[e(\varphi,h) = \underset{N \to \infty}{\lim} \frac{1}{N+1} \underset{n=0}{\overset{N}{\sum}}\varphi(X_n) - \inte{\mathcal{M}} \varphi(y) \deriv \mu_\infty (y) \]
 and since $\suite{X_n}{n}$ is ergodic, we have
 \[\underset{N \to \infty}{\lim} \frac{1}{N+1} \underset{n=0}{\overset{N}{\sum}} \varphi(X_n) = \inte{\mathcal{M}} \varphi(y) \rho^h(y) \deriv \vol_\mathcal{M}(y).\]
 Our aim is to show that
 \begin{align*} \inte{\mathcal{M}} \varphi(y) \rho^h(y) \deriv \vol_\mathcal{M}(y) = & \inte{\mathcal{M}} \varphi(y) \deriv \mu_\infty (y) \\ + & h^r \int_{0}^{\infty} \inte{\mathcal{M}} u(t,x) A_{r+1}^* \deriv \mu_\infty(x) + \mathcal{O}(h^{r+1}). \end{align*}
 Lemma~\ref{lemma:sequence} gives a sequence $\suite{\rho_n}{n}$ such that $\rho_0 = \rho_\infty$, and for all $N \geq 1$
 \[ \inte{\mathcal{M}} \rho_N(y) \deriv \vol_\mathcal{M}(y) = 0 \text{ and } \mathcal{L}^* \deriv \mu_N = \underset{l=1}{\overset{N}{\sum}} L_l^* \deriv \mu_{N-l}. \]
 Moreover, we have
 \[\abs{ \inte{\mathcal{M}} \varphi(y) \rho^h(y) \deriv \vol_\mathcal{M}(y) - \inte{\mathcal{M}} \varphi(y) \rho_N^h(y) \deriv \vol_\mathcal{M}(y) } \leq C h^{N+1}.\]
 For all $n=1,\dots,r-1$, $ \mathcal{L}^*\deriv \mu_n = A_n^* \deriv \mu_\infty = L_n^*
 \deriv \mu_\infty = 0 $ and $ \rho_n = 0 $ stands by induction. Thus the inequality
 of Lemma~\ref{lemma:second} ensures that
 \[ \abs{ \inte{\mathcal{M}} \varphi(y) \rho^h(y) \deriv \vol_\mathcal{M}(y) - \inte{\mathcal{M}} \varphi(y) \rho_r^h(y) \deriv \vol_\mathcal{M}(y) } \leq C_r \norm{\varphi}_{C^{8r+10}} h^{r+1}.\]
 However \(\rho_r^h(y) = \rho_\infty + h^r \rho_r\) and
 \[ \abs{ e(\varphi,h) - h^r \inte{\mathcal{M}} \varphi(y) \rho_r(y) \deriv \vol_\mathcal{M}(y) } \leq C_r \norm{\varphi}_{C^{8r+10}} h^{r+1}. \]
 The integration of equation~\eqref{equation:Kolmogorov}, ensures $u(t,x) =
 \varphi(x) + \int_0^t \mathcal{L} u (s,x) \deriv s$ and by ergodicity
 \[ \underset{t \to \infty}{\lim}u(t,x) = \inte{\mathcal{M}} \varphi(y) \deriv \mu_\infty (y). \]
 So by using that the mean of $\rho_r$ is zero and that, by ergodicity and equation~\eqref{equation:lim_u},
 \[ \varphi(x) = \inte{\mathcal{M}} \varphi(y) \deriv \mu_\infty (y) - \int_0^\infty \mathcal{L}u(s,x) \deriv s, \]
 we obtain
 \begin{align*}
 \inte{\mathcal{M}} \varphi(x) \rho_r(x) \deriv \vol_\mathcal{M}(x) = & - \inte{\mathcal{M}} \int_0^\infty \mathcal{L}u(s,x) \rho_r(x) \deriv s \deriv \vol_\mathcal{M}(x) \\
 = & - \int_0^\infty \inte{\mathcal{M}} u(s,x) \mathcal{L}^*\deriv \mu_r(x) \deriv \vol_\mathcal{M}(x) \deriv s \\
 = & \int_0^\infty \inte{\mathcal{M}} u(s,x) A_{r+1}^*\deriv \mu_\infty(x) \deriv s,
 \end{align*}
 where $\mathcal{L}^*\deriv \mu_r = A_{r+1}^* \deriv \mu_\infty$ as in the induction.
\end{proof}
We conclude this section by the proof of Theorem~\ref{theorem:theorem_post}.
\begin{proof}[Proof of Theorem~\ref{theorem:theorem_post}]
 Theorem~\ref{theorem:theorem_post} is proven by applying Theorem~\ref{theoreme:characterization} with \(\varPhi = \varphi \circ G\) where \(\varphi\) is a test function and \(G\) is the post-processor. First, we compute the term of error of the composition.
 \begin{align*}
 \inte{\MM} \varPhi(y) \rho_\infty(y) \deriv \vol_\MM (y) = & \inte{\MM} \varphi(y) \rho_\infty(y) \deriv \vol_\mathcal{M}(y) + \underset{j=1}{\overset{p-1}{\sum}} h^j \inte{\MM} \overline{A_j} \varphi(y) \rho_\infty(y) \deriv \vol_\mathcal{M}(y) \\
 + & h^p \inte{\MM} \overline{A_p} \varphi(y) \rho_\infty(y) \deriv \vol_\mathcal{M}(y) + \mathcal{O}(h^{p+1}),
 \end{align*}
 however by Assumption~\ref{assumption:H1}, the sum is zero.
 Hence by equation~\eqref{equation:lim_u}, it follows
 \begin{align*}
 \inte{\MM} \overline{A}_p \varphi(y) \deriv \mu_\infty (y) = & \left( \inte{\MM} \left( \inte{\MM} \varphi(z) \deriv \mu_\infty (z) \right) \overline{A}_p^*  \deriv \mu_\infty(y)  \right) \\
 & - \inte{\mathcal{M}} \left( \int_{0}^{\infty} \mathcal{L} u(y,t) \deriv t \right) \overline{A}_p^*  \deriv \mu_\infty(y)  \\
 = & - \int_{0}^{\infty} \inte{\mathcal{M}} u(y,t) \left(-\crochet{\mathcal{L}}{\overline{A}_p}^*\right) \deriv \mu_\infty (y) \deriv t \\
 = & - \int_{0}^{\infty} \inte{\mathcal{M}} u(y,t) A_{p+1}^*  \deriv \mu_\infty(y)  \deriv t
 \end{align*}
 where hypotheses \({(A_{p+1}+[\mathcal{L},\overline{A_p}])}^*\deriv \mu_\infty=0\) and \(\mathcal{L}^*\deriv \mu_\infty = 0\) are used. We remark that is it the term of order \(p\) in the error of Theorem~\ref{theoreme:characterization}. Now, Theorem~\ref{theoreme:characterization} can be applied with \(\varPhi = \varphi \circ G\) and a computation shows that
 \begin{align*}
 \underset{N \to \infty}{\lim} \frac{1}{N+1} \underset{n=0}{\overset{N}{\sum}} \varphi(\overline{X}_n) - \inte{\mathcal{M}} \varphi(y) \deriv \mu_\infty (y) = & \mathcal{O}(h^{p+1}).
 \end{align*}
 Thus \(\suite{\overline{X}_n}{n}\) is of order \(p+1\).
\end{proof}
\section{Intrinsic order conditions for the invariant measure}\label{Section:Intrinsic order conditions with planar exotic forests}
In this section, we use the algebraic framework of exotic Lie-Butcher series introduced in~\cite{Bronasco25hoi} to derive the order conditions for sampling the invariant measure. After recalling the notion of planar exotic forests, we present a new operation \(\IBP\) on forests, derived from the Riemannian integration by parts, to express the adjoint of the operators \(A_n\) and generalise~\cite{Laurent20eab, Bronasco22ebs}.
The approach is successfully applied through Theorem~\ref{theoreme:characterization} and Theorem~\ref{theorem:theorem_post} to derive the conditions of Theorem~\ref{theorem:conditions2}. We then discuss the form of the exotic Lie-Butcher series describing the adjoint of the operators \(A_n\) with irreducible forests.

\subsection{Planar exotic forests}\label{sub:Exotic forests Algebra}
We recall the structure of planar exotic forests, in the spirit
of~\cite{Bronasco25hoi} which generalize standard B-series~\cite{Butcher72aat, Hairer06gni, Butcher21bsa} to SDEs on manifolds.
\begin{definition}
 A decorated tree \((\pi,\alpha)\) with the decoration set \(D\) is a connected oriented graph \(\pi = (V,E)\) with the vertices $V$ and the edges $E\subset V\times V$, in which each node has exactly one outgoing edge except for one node, called the root, that has none. The map \(\alpha : V \to D\) decorates each vertex. A decorated tree is called planar if for all node \(v \in V\), the set \(\Pi(v)\) of predecessors of \(v\) is ordered. An ordered, possibly empty, list of planar decorated trees is called a planar decorated forest.

 An exotic forest is a planar decorated forest with the decorations \(D =
 \{\bullet\} \cup \mathds{N}\), \(\mathds{N} = \{1,2,\ldots\}\) which follows
 the following rules: if an integer is used as decoration then it must decorate
 exactly two leaves. Two exotic trees \((\pi_1,\alpha_1)\) and
 \((\pi_2,\alpha_2)\) are identical if \(\pi_1=\pi_2\) and if there exists an
 application \(\psi : D \mapsto D\) such that \(\psi(\bullet) = \bullet\),
 \(\restr{\psi}{\mathds{N}}\) induces an automorphism of \(\mathds{N}\) and
 \(\alpha_1 = \psi \circ \alpha_2\). Define \(V_{\forestL}\) as the set of
 nodes of \(\pi\) decorated by \(\forestL\) and \(V_L = \{ \{v,w\} \subset V | \alpha_w(v) = \alpha_w(w) \} \) as the set of pairly decorated nodes, that we call lianas.
 The set of exotic forests is \(\EF\), the associated vector space is \(\EFv=\Span_\R(\EF)\).

 The order of a forest \(\pi\in \EF\) is defined by
 \[\abs{\pi} =\abs{ V_{\forestL} } + \abs{V_L}. \]
 We denote \(\EFv_{k}\) the vector space spanned by exotic forests of order equal to \(k\), and \(\EFv_{\leq k}\) the vector space spanned by exotic forests of order lower than or equal to \(k\).
\end{definition}
\begin{ex}
 The forests of order 1 are $\forestL$ and $\forestM$. The space \(\EFv_{2}\) is generated by the following 11 forests:
 \[ \forestA, \quad \forestB, \quad \forestC, \quad \forestD, \quad \forestE, \quad \forestF, \quad \forestG, \quad \forestH, \quad \forestI, \quad \forestJ, \quad \forestK.\]
 There are 95 forests of order 3.
\end{ex}
\begin{definition}
 Let the map \(\mathds{F}^F : \EFv \to T(\mathcal{\mathfrak{X}}_P(\MM)) \) be defined for \(\pi \in \EF\) and \(\varphi\in C_P^\infty(\MM)\) by
 \[ \mathds{F}^F (\pi) \vartriangleright \varphi = \underset{v \in V}{\underset{i_v = 1}{\overset{D}{\sum}}} \left( \underset{\{v,w\} \in V_L}{\prod} \delta_{i_v=i_w} \right) \left( \underset{v \in V_{\forestL}}{\prod} E_{I_{\Pi(v)}} \left[ f^{i_v} \right] \right) E_{I_{R}} [\varphi], \]
 where \(R\) is the set of roots of \(\pi\), \(\Pi(v)\) is the set of
 predecessors of \(v\), both ordered from right to left, \(i_v\) is the indice
 associated to the node \(v\) then and \(E_I[g] = E_{i_p}[\ldots E_{i_1}[g]
 \ldots]\) if the set of indices is \(I = (i_p,\ldots,i_1)\). If \(\{v,w\}\) is a liana,
 i.e.\ts, they share the same integer as decoration, then the Kronecker symbol $\delta_{i_v=i_w}$ ensures that they share the same indice, i.e.\ts, \(i_v=i_w\).
\end{definition}
 From~\cite{Bronasco25hoi}, the operator \(\mathds{F}^F: (\EFv,\cdot) \to (T(\mathcal{\mathfrak{X}}_P(\MM)),\cdot) \) is a Hopf algebra morphism. In particular, $\mathds{F}^F$ satisfies
 \[
 \mathds{F}^F(\pi_1\cdot \pi_2)=\mathds{F}^F(\pi_1)\cdot \mathds{F}^F(\pi_2).
 \]
\begin{definition}\label{definition:Sseries}
    Given a one form \(a \in \EF^*\), called a coefficient map, an exotic Lie S-series is the following formal power series in \(h\) in \(T(\mathcal{\mathfrak{X}}_P(\MM))\),
    \[S_h(a) \vartriangleright \varphi = \underset{\pi \in \EF}{\sum} h^{\abs{\pi}} a(\pi) \mathds{F}^F(\pi) \vartriangleright \varphi.\] 
    \end{definition}
    In the following, we denote by \(a_w\) the coefficient map of the numerical flow, that is the unique coefficient map \(a_w : \EFv \to \mathds{R}\) such that
    \[ \espe{\varphi(X_1)\vert X_0 = x} = S_h(a_w) \vartriangleright \varphi .\]

\subsection{Integration by parts for the invariant measure}\label{sub:Integration by parts for the invariant measure}

The derivation of the order conditions for the invariant measure is performed using multiple integration by parts of the Taylor-Talay-Tubaro expansion against the invariant measure.
In this context, the use of an orthonormal frame proves crucial to ensure that the operations rewrite with planar exotic forests.
\begin{proposition}\label{proposition:integrationbypart}
 For \(\varphi,\psi \in C_P^\infty(\MM)\), the following integration by parts with respect to the Gibbs measure~\eqref{equation:Gibbs} holds:
 \[ \inte{\MM} E_i[\varphi] \psi \deriv \mu_\infty = - \inte{\MM} \varphi E_i[\psi] \deriv \mu_\infty - \inte{\MM} \varphi \psi f^i \deriv \mu_\infty, \]
 where \(F = f^j E_j\) is given by~\eqref{equation:correction}.
\end{proposition}
\begin{proof}
 We recall the general integration by part's formula~\cite{ONeill83srg, Lee19irm}, for \(X = x^i E_i \in \mathfrak{X}_P(\MM)\),
 \[\inte{\mathcal{M}} X[\varphi] \deriv \vol_\mathcal{M} = - \inte{\mathcal{M}} \varphi \Div(X) \deriv \vol_\mathcal{M} = - \inte{\mathcal{M}} \varphi \left( x^i \Div(E_i) + E_i [x^i] \right) \deriv \vol_\mathcal{M}.\]
 Then, by applying the formula to the invariant measure with \(X = \psi \rho_\infty E_i\), we have
 \begin{align*}
 \inte{\mathcal{M}} E_i[\varphi] \psi \deriv \mu_\infty = & - \inte{\mathcal{M}} \varphi \left( \psi \Div(E_i) + E_i[\psi] - \psi E_i[V] \right) \deriv \mu_\infty \\
 = & - \inte{\mathcal{M}} \varphi E_i[\psi] \deriv \mu_\infty - \inte{\mathcal{M}} \varphi \psi \left(\Div(E_i) - E_i[V]\right) \deriv \mu_\infty.
 \end{align*}
 Proposition~\ref{proposition:integrationbypart} holds if \(\Div(E_i) - E_i[V] = f^i\).
 Recall the correction term~\eqref{equation:correction} that is \(f^i = -E_i[V] -g(\nabla_{E_d}E_d , E_i)\) in the orthonormal basis \(\{E_1,\dots,E_D\}\). Hence, Proposition~\ref{proposition:integrationbypart} holds if \(\Div (E_i) = - g(\nabla_{E_d}E_d , E_i)\). However, by definition of the operator \(\Div(H) = g(\nabla_{E_d}H , E_p)\) for all \(H\in \mathfrak{X}_P(\MM)\), the Koszul formula~\cite{ONeill83srg,Lee19irm} ensures that
 \begin{align*}
 2 g(\nabla_{E_d} E_i, E_d) = & E_d \bigl(g(E_i,E_d)\bigr) + E_i \bigl(g(E_d,E_d)\bigr) - E_d \bigl(g(E_d,E_i)\bigr) \\
 + & g([E_d,E_i],E_d) - g([E_i,E_d], E_d) - g([E_d,E_d], E_i) \\
 = & 2 g([E_d,E_i],E_d).
 \end{align*}
Similarly, one finds \(2 g(\nabla_{E_d}E_d , E_i) = - 2 g([E_d,E_i],E_d) \). Hence the result.
\end{proof}
\begin{ex}
 Proposition~\ref{proposition:integrationbypart} yields for \(\varphi = E_d[\phi]\) and \(\psi \equiv 1\).
 \[ \inte{\mathcal{M}} ((E_d \cdot E_d) \vartriangleright \phi) \deriv \mu_\infty = - \inte{\mathcal{M}} f^d (E_d \vartriangleright \phi) \deriv \mu_\infty. \]
 We thus recover \[\mathcal{L}^* \deriv \mu_\infty = {\left( f^d E_d + (E_d \cdot E_d) \right)}^* \deriv \mu_\infty = 0.\]
\end{ex}
\begin{remark}
In general, Proposition~\ref{proposition:integrationbypart} does not yield elementary differentials that can be represented as exotic forests. For \(\psi = f^i\), the terms \(E_i[f^i]\) and $f^i f^i$ appear, which cannot be represented by exotic forests.
 Such terms are called aromas and yield the larger space of exotic aromatic forests in the Euclidean case~\cite{Laurent21ocf, Laurent23tue, Bronasco22cef}. The use of aromas is out of the scope of the present paper.
\end{remark}

Let us now identify the exotic forests whose integration by part does not write with on exotic forests.
\begin{definition}
 The set of irreducible forests \( \Irred\) is the set of exotic forest whose first tree is not a numbered node.
 If \( \pi \notin \Irred \) and its first root is numbered by \( \forestN \), we write \(\pi = \forestN \tilde{\pi}\) where \(\tilde{\pi}\) is a decorated forest.
 We denote \( \Irredv = \Span_\R (\Irred)\).
\end{definition}

The integration by parts of planar exotic forests is defined as the following, which generalises the non-planar map in~\cite{Bronasco22ebs}.
\begin{definition}\label{definition:IBP}
 Define the linear operator \(\IBP : \EFv \to \EFv\) by \(\restr{\IBP}{\Irred} \equiv 0 \) and else
 \begin{equation}
 \label{equation:IBP_def}
 \IBP(\pi ) = - \restr{\tilde{\pi}}{\forestN \to \forestL} - \underset{v \in V_{\forestL}}{\sum} \restr{\tilde{\pi}}{\forestN \curvearrowright v},
 \end{equation}
 where \(\restr{\tilde{\pi}}{\forestN \to \forestL}\) is the exotic
 forest obtained by substituting the unique node \( \forestN \) of \(\tilde{\pi}\)
 by \( \forestL \) and \(\restr{\tilde{\pi}}{\forestN \curvearrowright v}\)
 is the exotic forest obtained by left-grafting a node \( \forestN \) on the node \(v\) of \(\tilde{\pi}\).
\end{definition}

\begin{remark}
It is well known that Rota-Baxter algebras~\cite{Baxter60aap} are linked to the representation of integration by parts operations.
The IBP operation does not exactly yield a Rota-Baxter structure thanks to the perturbation $\restr{\tilde{\pi}}{\forestN \to \forestL}$ in equation~\eqref{equation:IBP_def}.
\end{remark}

\begin{exs}
 A direct application of Definition~\ref{definition:IBP} yields:
 \begin{align*}
 & \IBP(\forestL + \forestM) = - \forestL, \quad \IBP(\forestE ) = -(\forestB + \forestA), \quad \IBP(\forestK) = - \forestH, \\
 & \IBP(\forestAA) = -(\forestAB+\forestAC+\forestAD).
 \end{align*}
\end{exs}

We now derive an explicit expression of the adjoint operator \(A_j^*\) with irreducible forests. Since \(\IBP\) removes only the first numbered root, we iterate \(\IBP\) until the output is a linear combination of irreducible forests.
\begin{definition}\label{definition:RED}
The linear reduction operator \(\RED\) on \(\EFv\) is given by the following limit of stationary sequence:
 \[ \RED = \underset{n \to \infty}{\lim} {\left( \id_{\Irredv} + \IBP \right)}^n.\]
\end{definition}
\begin{exs}
A direct application of Definition~\ref{definition:RED} yields:
 \begin{align*}
 & \RED(\forestL + \forestM) = 0, \quad \RED(\forestE ) = -(\forestB + \forestA), \quad \RED(\forestK) = \forestC - \forestB - \forestA, \\
 & \RED(\forestAA) = \forestAE + \forestAF + \forestAG + \forestAH + \forestAI + \forestAJ + \forestAK + \forestAL + \forestAM.
 \end{align*}
\end{exs}
\begin{definition}
 Two exotic forests \(\pi_1, \pi_2 \in \EFv\) are equivalent, written \(\pi_1 \sim \pi_2\), if the associated differential operators satisfy
 \[ \mathds{F}^F {(\pi_1)}^*\deriv \mu_\infty = \mathds{F}^F {(\pi_2)}^*\deriv \mu_\infty. \]
\end{definition}
\begin{proposition}\label{proposition:equivalence}
 For all \(\pi \in \EFv\), \(\pi\) and \(\RED(\pi)\) represent the same adjoint operator, that is \[ \pi \sim \RED(\pi) .\]
\end{proposition}

The algebraic characterization of Theorem~\ref{theoreme:characterization} rewrites naturally with planar exotic forests.
\begin{theorem}\label{theorem:cancelling}
 If the coefficient map \(a_w\) of a numerical method satisfies 
 \[\underset{\pi \in \EF_p}{\sum} a_w(\pi) \RED(\pi) = 0\]
  then its Talay-Tubaro expansion satisfies \(A_p^* \deriv \mu_\infty = 0\).
\end{theorem}
In order to compute the second order condition from Theorem~\ref{theorem:conditions2}, Table~\ref{table:order2} gives the reduction of second order exotic forests.
\begin{table}[H]
 \centering
 \begin{tabular}{ c | c } 
 {Forest} \(\pi\) & \(\RED(\pi)\) \\
 \hline 
 $\forestA$ & $\forestA$ \\
 $\forestB$ & $\forestB$ \\
 $\forestC$ & $\forestC$ \\
 $\forestD$ & $\forestD$ \\
 $\forestE$ & $-\forestA - \forestB$ \\
 $\forestF$ & $\forestF$ \\
 $\forestG$ & $-\forestC - \forestD$ \\
 $\forestH$ & $ \forestA + \forestB -\forestC$ \\
 $\forestI$ & $-\forestF$ \\
 $\forestJ$ & $ \forestC + \forestD$ \\
 $\forestK$ & $ \forestC - \forestA - \forestB$ \\
 \end{tabular}
 \caption{\scriptsize Reduction of exotic forests of order 2.}\label{table:order2}
\end{table}
\begin{proof}[Proof of Theorem~\ref{theorem:conditions2}]
 By linearity and Proposition~\ref{proposition:equivalence}, the computation of Table~\ref{table:order2} reduces the second order of the error for the invariant measure as
 \begin{align*}
 \underset{\pi \in \EF_2}{\sum} a_w(\pi) \RED (\pi) = & a \left(\forestA - \forestE + \forestH - \forestK \right) \forestA \\
 + & a \left(\forestB - \forestE + \forestH - \forestK \right) \forestB \\
 + & a \left(\forestC - \forestG - \forestH + \forestJ + \forestK \right) \forestC \\
 + & a \left(\forestD - \forestG + \forestJ \right) \forestD \\
 + & a \left(\forestF - \forestI \right) \forestF.
 \end{align*}
 Hence, by Theorem~\ref{theorem:cancelling}, cancelling each coefficients ensures that \(A_2^*\deriv \mu_\infty = 0\). Rewriting these conditions gives the second order conditions from Theorem~\ref{theorem:conditions2}.
\end{proof}

\subsection{Exotic S-series expression of the adjoint operators}\label{sub:Exotic S-series expression of the adjoint operators}
Let us distinguish the coefficient maps for the weak error \(a_w : \EFv \to \mathds{R}\) and for the invariant measure.
Denote by \(\ps{\cdot}{\cdot}\) the bilinear pairing over \(\EFv\) such that, for \(\pi_1, \pi_2 \in \EF\),
\[ \ps{\pi_1}{\pi_2} = 1 \text{ if } \pi_1 = \pi_2 \neq 0 \text{ and } 0 \text{ else}.\]
\begin{definition}\label{definition:a_mu}
 Let \(A : \Irredv \to \EFv\) be the operator
 \[ A : \pi \mapsto {\underset{\tilde{\pi} \in \EF}{\sum}} \ps{\pi}{\RED(\tilde{\pi})}\tilde{\pi} = \underset{k \geq 0}{\underset{\tilde{\pi} \in \EF}{\sum}} \ps{\pi}{{\IBP}^{k}(\tilde{\pi})}\tilde{\pi} .\]
For the coefficient map of a numerical method \(a_w\), let us define \(a_{\mu_\infty} = a_w \circ A\).
\end{definition}
\begin{remark}
 The operator \(A\) is invertible and \(A^{-1} = \id - {\IBP}^{*}\).
\end{remark}

\begin{proposition}
 The map \(a_{\mu_\infty} : \Irredv \to \mathds{R}\) is the coefficient map for the invariant measure, that is, the error~\eqref{equation:error} for the invariant measure of a method of order p is expressed using a S-series indexed only on irreducible forests of order \(p\),
 \[ e(\varphi,h) = h^p \int_{0}^{\infty} \inte{\mathcal{M}} u(t,x) {\big( \underset{\abs{\pi} = p}{\underset{\pi \in \Irred}{\sum}} a_{\mu_\infty}(\pi) \mathds{F}^F (\pi)\big)}^* \deriv \mu_\infty(x) \deriv t.\]

\end{proposition}
\begin{exs}
 Using Table~\ref{table:order2}, the computation of \(a_{\mu_\infty}\) gives for the first and second orders:
 \begin{align*}
 a_{\mu_\infty} ( \forestL ) & = a_w (\forestL - \forestM)\\
 a_{\mu_\infty} \left(\forestA \right) & = a_w \left(\forestA - \forestE + \forestH - \forestK \right), \\
 a_{\mu_\infty} \left(\forestB \right) & = a_w \left(\forestB - \forestE + \forestH - \forestK \right) , \\
 a_{\mu_\infty} \left( \forestC \right) & = a_w \left(\forestC - \forestG - \forestH + \forestJ + \forestK \right), \\
 a_{\mu_\infty} \left(\forestD \right) & = a_w \left(\forestD - \forestG + \forestJ \right) \\
a_{\mu_\infty} \left(\forestF \right) & = a_w \left( \forestF - \forestI \right).
 \end{align*}
\end{exs}
It is important to mention that the conditions for the invariant measure are not independent and satisfy relations, analogously to the weak context~\cite{Bronasco25hoi}. However, the map $a_{\mu_\infty}$ is not a character for the shuffle product as one has for instance
 \[ {a_{\mu_\infty}(\forestL)}^2 = 2 a_{\mu_\infty}(\forestC - \forestF) , \quad {a_{w}(\forestL)}^2 = 2 a_{w}(\forestA).\]

\begin{proposition}
 The coefficient map \(a_{\mu_\infty}\) is a character for the modified shuffle product,
 \[ \tilde{\shuffle} : (\pi_1,\pi_2) \mapsto A^{-1}(A\pi_1 \shuffle A\pi_2) = (A\pi_1 \shuffle A\pi_2) - {\IBP}^{*}(A\pi_1 \shuffle A\pi_2) . \]
\end{proposition}
\begin{proof}
 Definition~\ref{definition:a_mu} ensures that for all \(\pi_1,\pi_2\in \Irred\),
 \begin{align*}
 a_{\mu_\infty}(\pi_1)a_{\mu_\infty}(\pi_2) = & a_w(A\pi_1)a_w(A\pi_2) \\
 = & a_w(A\pi_1 \shuffle A \pi_2) \\
= & a_{\mu_\infty}(A^{-1}\left( A\pi_1 \shuffle A \pi_2 \right)),
 \end{align*}
 where we used that $a_w$ is a character for $\shuffle$ (see~\cite{Bronasco25hoi}).
 Hence the result.
\end{proof}
\begin{theorem}\label{theorem:order k}
    Under the assumptions of Theorem~\ref{theoreme:characterization} (or Theorem~\ref{theorem:theorem_post} for post-processed methods) if the coefficient map for the invariant measure of a consistent numerical method satisfies
 \[ \restr{a_{\mu_\infty}}{\Irred_{\leq k}} \equiv 0 ,\]
 then the method is of order \(k\) for the invariant measure.
\end{theorem}

\begin{remark}
    If a method is consistent then \(\restr{a_{\mu_\infty}}{\Irred_{\leq 1}} \equiv 0\).
\end{remark}

\section{Numerical experiments}\label{section:Numerical experiments}
\subsection{Ergodicity and long-time sampling on \(SO_p\)}\label{sub:Experiments on SO_3}
Our first experiment focuses on \(\MM=SO_p\) the compact Lie group of special orthogonal matrices of size \(p \times p\). Let \({\left(A_d\right)}_{d=1,\ldots,D}\) be an orthonormal basis of its Lie algebra \(\mathfrak{so}_p\) of skew-symmetric matrices for the metric \(g_x(A,B) = \Trace \left( A^T B \right)\), \(A,B \in T_x \MM\). Let \(E_d(y) = A_d y\) be the associated orthonormal frame basis. As the manifold is smooth and compact, the assumptions of Theorem~\ref{theorem:theorem_post} are satisfied and our analysis applies.
\begin{remark}
 On a matrix space with \(A_1,\dots,A_D\) as a basis of the Lie algebra, the frame is \(E_d(p) = A_d p\), the geodesic flow \(\exp\) is given by the matrix exponential \(\Exp\), and Method~\ref{method:pp} becomes the following.
 \begin{algorithm}[H]
 \caption*{\textbf{Method 1} Post-processed frozen-flow method for equation~\eqref{equation:Langevin} on $SO_p$}
 \begin{algorithmic}
 \STATE$H_n = \Exp \left( \frac{\sqrt{2h}}{2} \xi^d_n  A_d\right) X_n$,
 \STATE$X_{n+1} = \Exp \left( \left(\frac{5h}{4} f^d(H_n) + \frac{\sqrt{2h}}{4} \xi^d_n\right) A_d \right)\Exp \left( \left(- \frac{h}{4} f^d(H_n) + \frac{3\sqrt{2h}}{4} \xi^d_n\right) A_d \right) X_n$,
 \STATE$\overline{X_N} = \Exp \left( \frac{\sqrt{2h}}{2} \xi^d_N A_d \right) X_N$.
 \end{algorithmic}
 \end{algorithm}
\end{remark}
Inspired by the experiments in~\cite{Bronasco25hoi}, we compare the new methods of high order for the invariant measure to the frozen-flow Euler method~\eqref{equation:frozenEuler} and the so-called SFF2 method from~\cite{Bronasco25hoi}, which uses two random variables and is of weak order 2.
We consider dynamics with the test function
\begin{equation}
\label{eq:num_test_fct}
\varphi(x) = \exp \left( - \frac{g_x (x-I_p,x-I_p)}{2p} \right). 
\end{equation}
The reference solution is taken as the mean value of the SFF2 method~\cite{Bronasco25hoi} with the reference time step \(h=2^{-12}\).
We compare the order of convergence in long time for two potentials of the form \(V = P( \norm{\cdot - I_p}^2 ) \) where \(P\) polynomial: a quadratic potential well with \(P = 10 X\) and a sextic potential with multiple local minima using the cubic polynomial \( P = \frac{X}{2} (X^2-9X+24)\).
Starting from the initial condition \(X_0 =  - \left(\begin{smallmatrix} 1&0&0\\0&0&1\\0&1&0 \end{smallmatrix}\right)\), we observe in Figure~\ref{fig:ergodique_quadratic} the ergodic behaviour of the dynamic as the trajectories explore a neighbourhood of the identity matrix after an initial transient phase.
\begin{figure}[H]
 \centering
 \begin{subfigure}[b]{0.44\linewidth}
 \includegraphics[width=\linewidth]{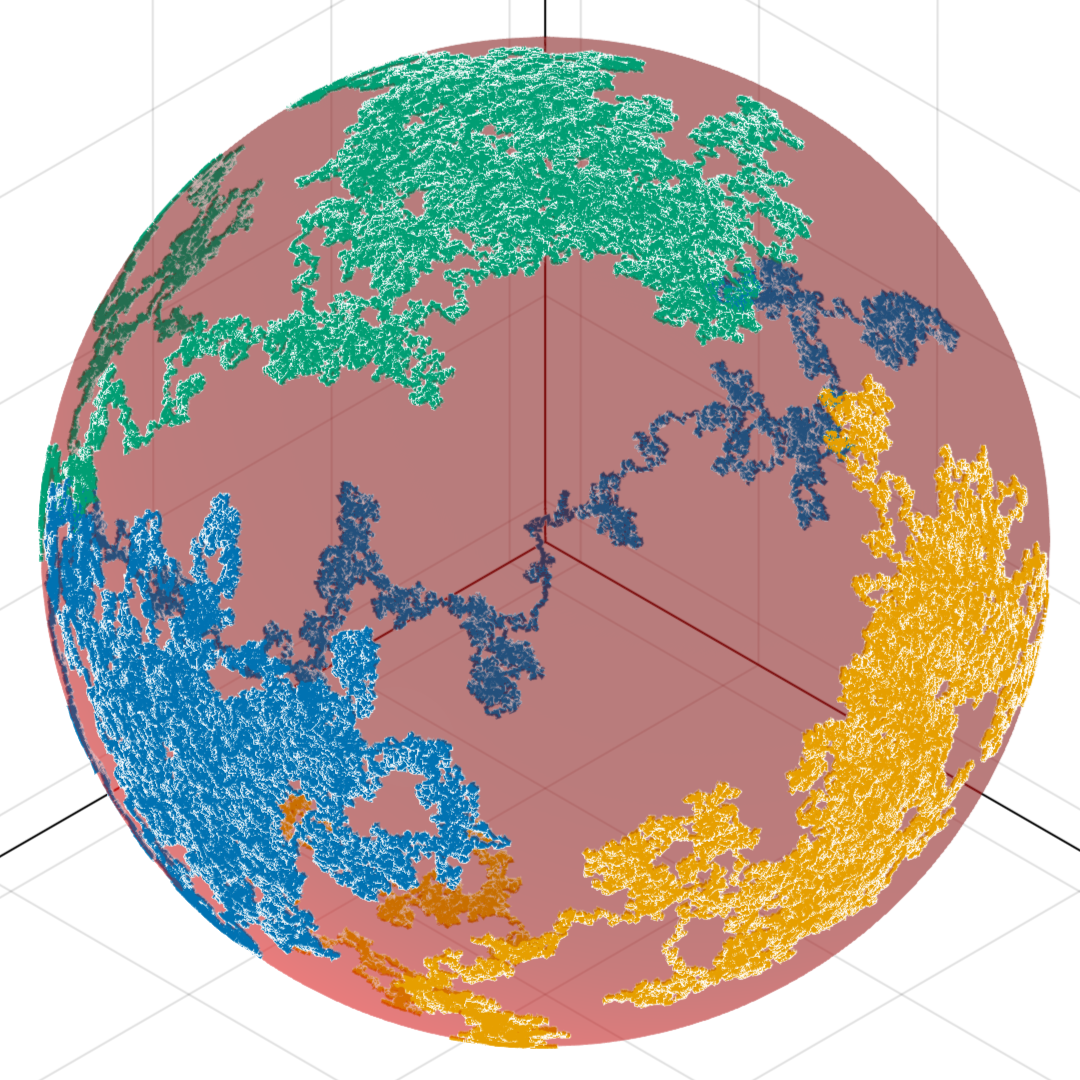}
 \end{subfigure}
 \hfill
 \begin{subfigure}[b]{0.54\linewidth}
 \includegraphics[width=\linewidth]{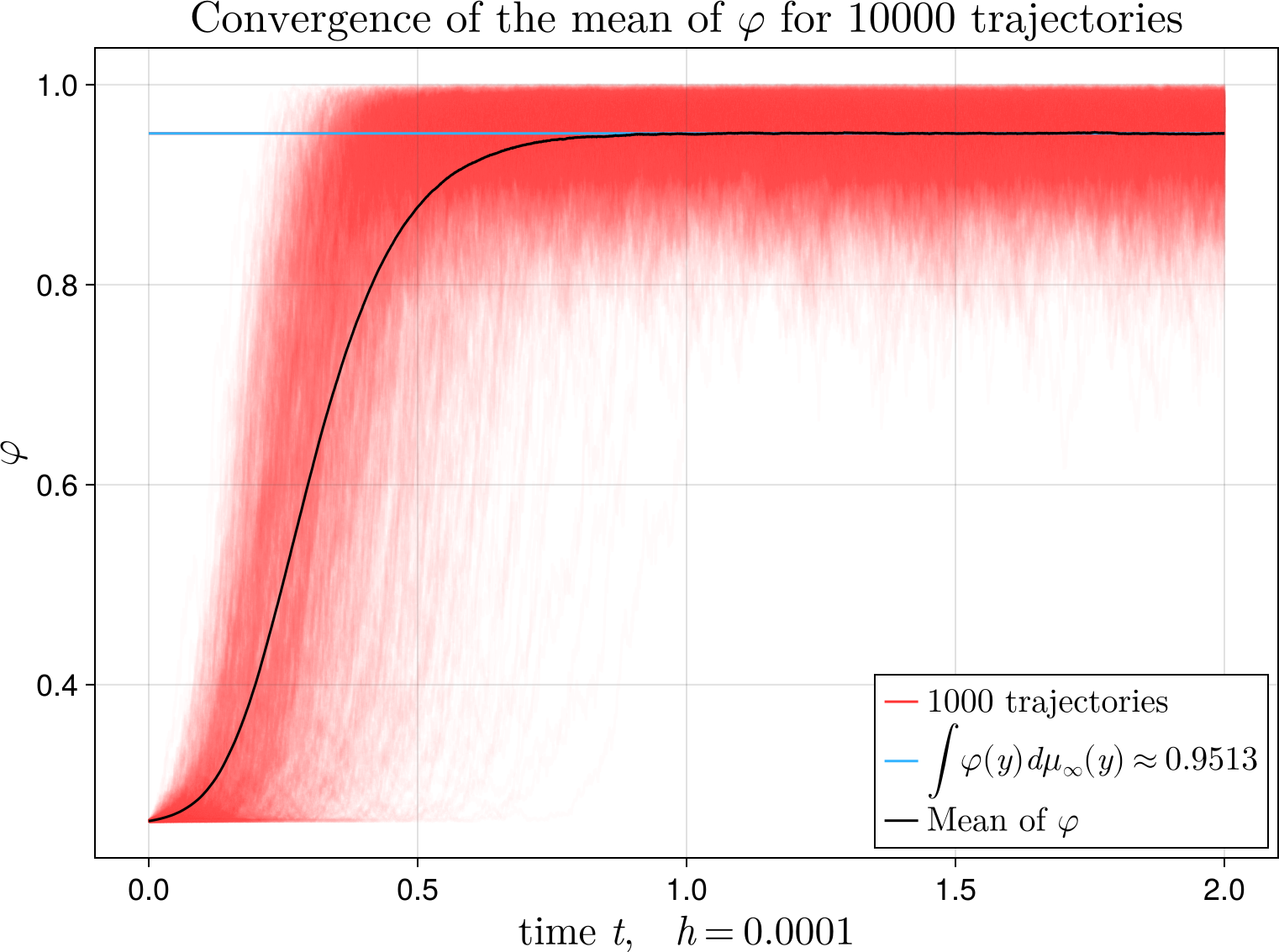}
 \end{subfigure}
 \caption{\scriptsize Trajectory on \(SO_3\) with quadratic potential at time \(T=2\) and step time \(h = 10^{-5}\), represented by the three columns on the sphere \(\mathds{S}^2\) \ (Left). Ergodic convergence computed with \(10^4\) trajectories and \(h=10^{-4}\) for the quadratic potential and the test function \(\varphi\)\ (Right). The method used is Method~\ref{methode1} in both computation. }\label{fig:ergodique_quadratic}
\end{figure}


We observe the error curves for the invariant measure in Figure~\ref{fig:ergodique_So}. We observe that the new methods~\ref{method:pp},~\ref{methode1} and~\ref{methode2} exhibit the expected second-order behavior, with a reduced cost compared to SFF2, which confirms our theoretical results. Moreover, Method~\ref{method:pp} rapidly reaches the Monte-Carlo threshold in the quadratic case, hinting that it may have a higher order of convergence for this specific potential.
\begin{figure}[H]
 \centering
 \begin{subfigure}[b]{0.49\linewidth}
 \includegraphics[width=\linewidth]{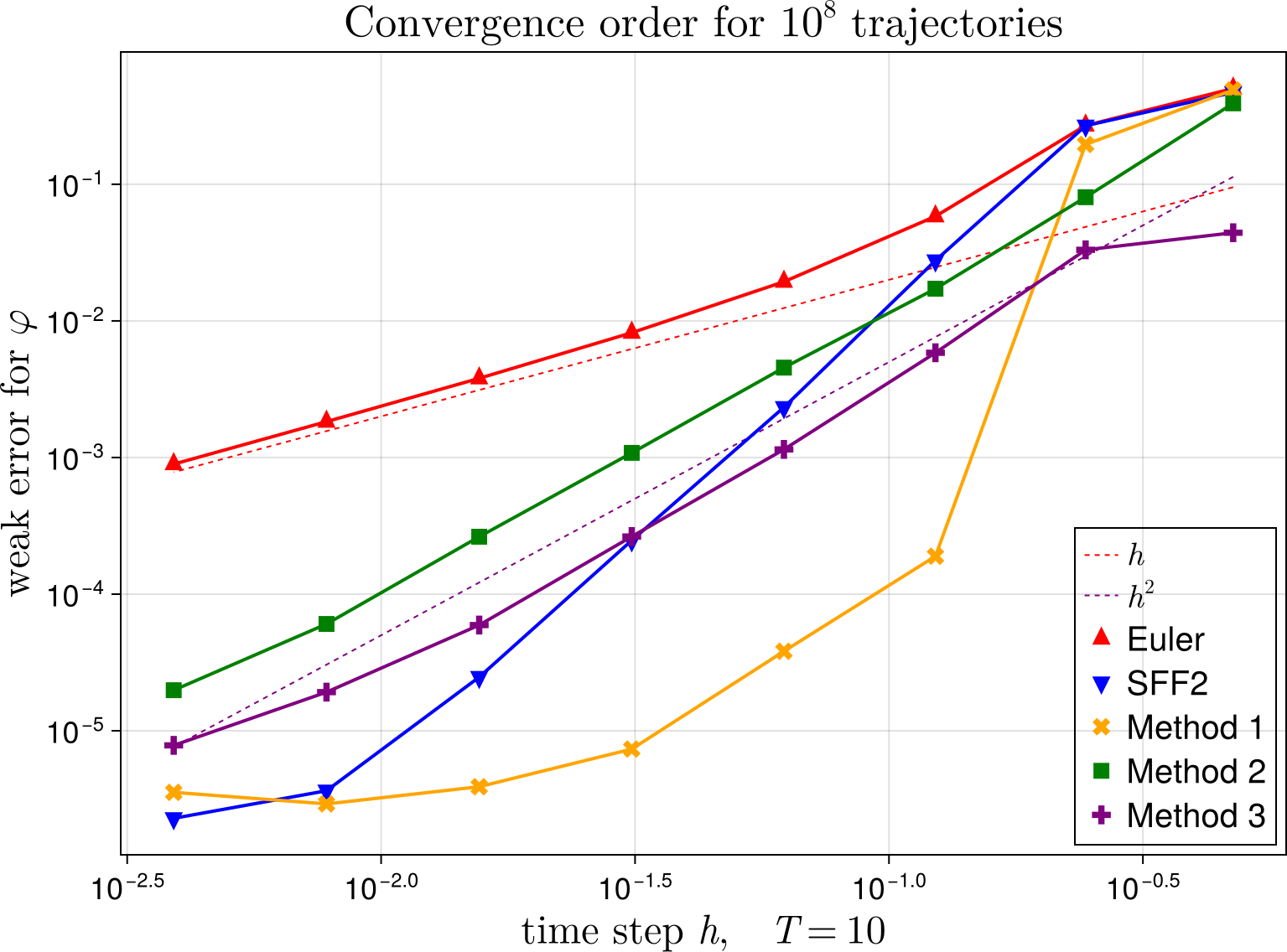}
 \end{subfigure}
 \hfill
 \begin{subfigure}[b]{0.49\linewidth}
 \includegraphics[width=\linewidth]{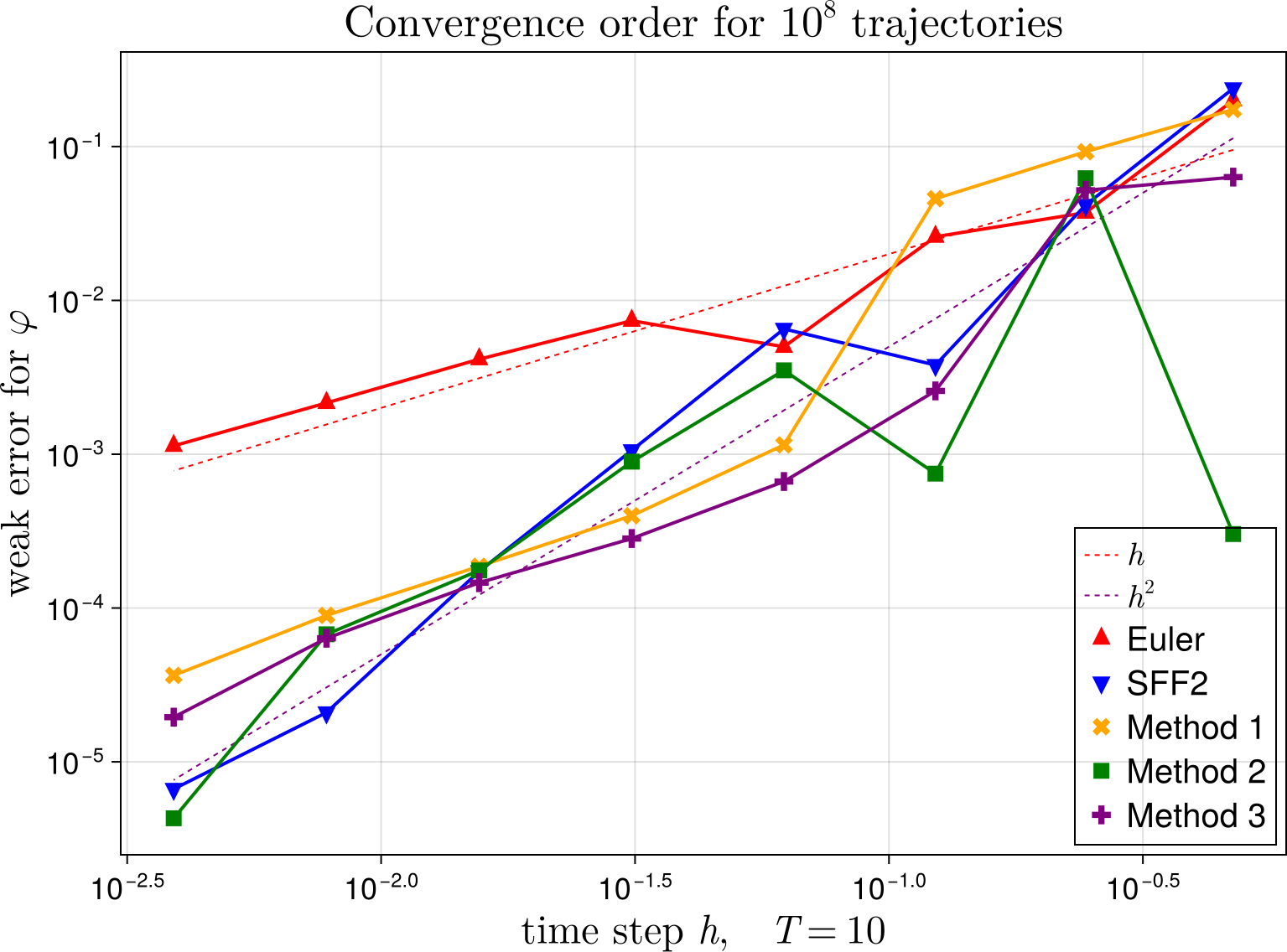}
 \end{subfigure}
 \caption{\scriptsize Order of convergence in long time for two potentials: quadratic (left) and sextic (right). The mean is computed at the final time \(T=10\) with \(10^8\) trajectories and with the test function \(\varphi\) given by~\eqref{eq:num_test_fct}.}\label{fig:ergodique_So}
\end{figure}

\subsection{Von-Mises Fisher dynamics on the sphere \(\mathds{S}^2\) }\label{sub:Experiments on S^2}

Our second experiment focuses on the 2-dimensional sphere \(\MM=\mathds{S}^2\). As our new methods rely on orthonormal bases, we consider the standard coordinates on \(\mathds{S}^2\) minus the poles  \( ( 0 , 0 , \pm 1) \):
\[ y_{\theta,\phi} = \begin{pmatrix} \cos (\theta) \cos (\phi) \\ \cos (\theta) \sin (\phi) \\ \sin (\theta) \end{pmatrix} \in \mathds{S}^2, \quad \theta \in \left(-\frac{\pi}{2},\frac{\pi}{2}\right), \quad \phi \in \left[0,2\pi\right), \]
with the associated orthonormal frame:
\[ E_1(y_{\theta,\phi}) = \begin{pmatrix} -\sin (\theta) \cos (\phi) \\ -\sin (\theta) \sin (\phi) \\ \cos (\theta) \end{pmatrix}, \quad  E_2(y_{\theta,\phi}) = \begin{pmatrix} - \cos (\theta) \sin (\phi) \\ \cos (\theta) \cos (\phi) \\ 0 \end{pmatrix}. \]
Using the symmetry \( S : (x,y,z) \mapsto (z,y,x) \), we define a second frame on \(\mathds{S}^2 \) minus \( ( \pm 1 , 0 , 0) \). For one step of the method, we use the first frame basis if $\abs{\sin(\theta)}\leq 0.6$, and the second frame else. This ensures a bounded Lipschitz constant for the chosen frame in each case and the assumptions of our analysis are satisfied.
Following the experiments of~\cite{Laurent21ocf,Bharath23sae,Bronasco25hoi}, we consider the potential  \(V : (x,y,z) \mapsto -25z \), the associated vector field (containing the Ito correction~\eqref{equation:correction})
\[F = - \nabla V - \nabla_{E_1} E_1 - \nabla_{E_2} E_2,\]
and the test function \(\varphi : (x,y,z) \mapsto z^2 \).
This choice of potential confines the trajectories around the north pole of \(\mathds{S}^2\) (see Figure~\ref{fig:ergodique_sphere}).
The reference solution is chosen as the mean value of the SFF2 method~\cite{Bronasco25hoi} with the reference time step \(h=\frac{1}{ 25 \times 2^{10}}\).

We compare the new methods of high order for the invariant measure with the methods of~\cite{Bronasco25hoi}. The error curves for the invariant measure in Figure~\ref{fig:ergodique_sphere} display the expected orders of convergence. In particular, Methods~\ref{methode1} and~\ref{methode2} exhibit second-order behavior, while Method~\ref{method:pp} reaches the Monte-Carlo threshold instantly.
\begin{figure}[H]
 \centering
 \begin{subfigure}[b]{0.44\linewidth}
 \includegraphics[width=\linewidth]{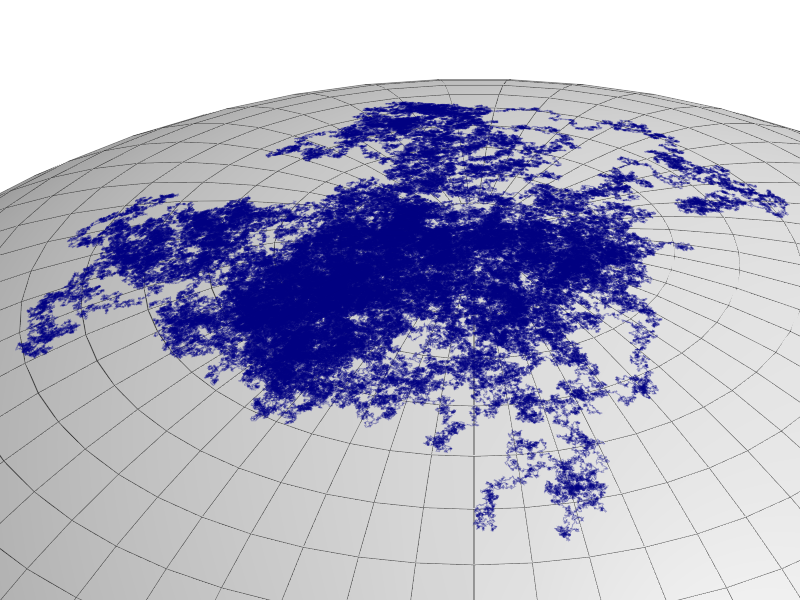}\label{fig:traj_Sphere}
 \end{subfigure}
 \hfill
 \begin{subfigure}[b]{0.54\linewidth}
 \includegraphics[width=\linewidth]{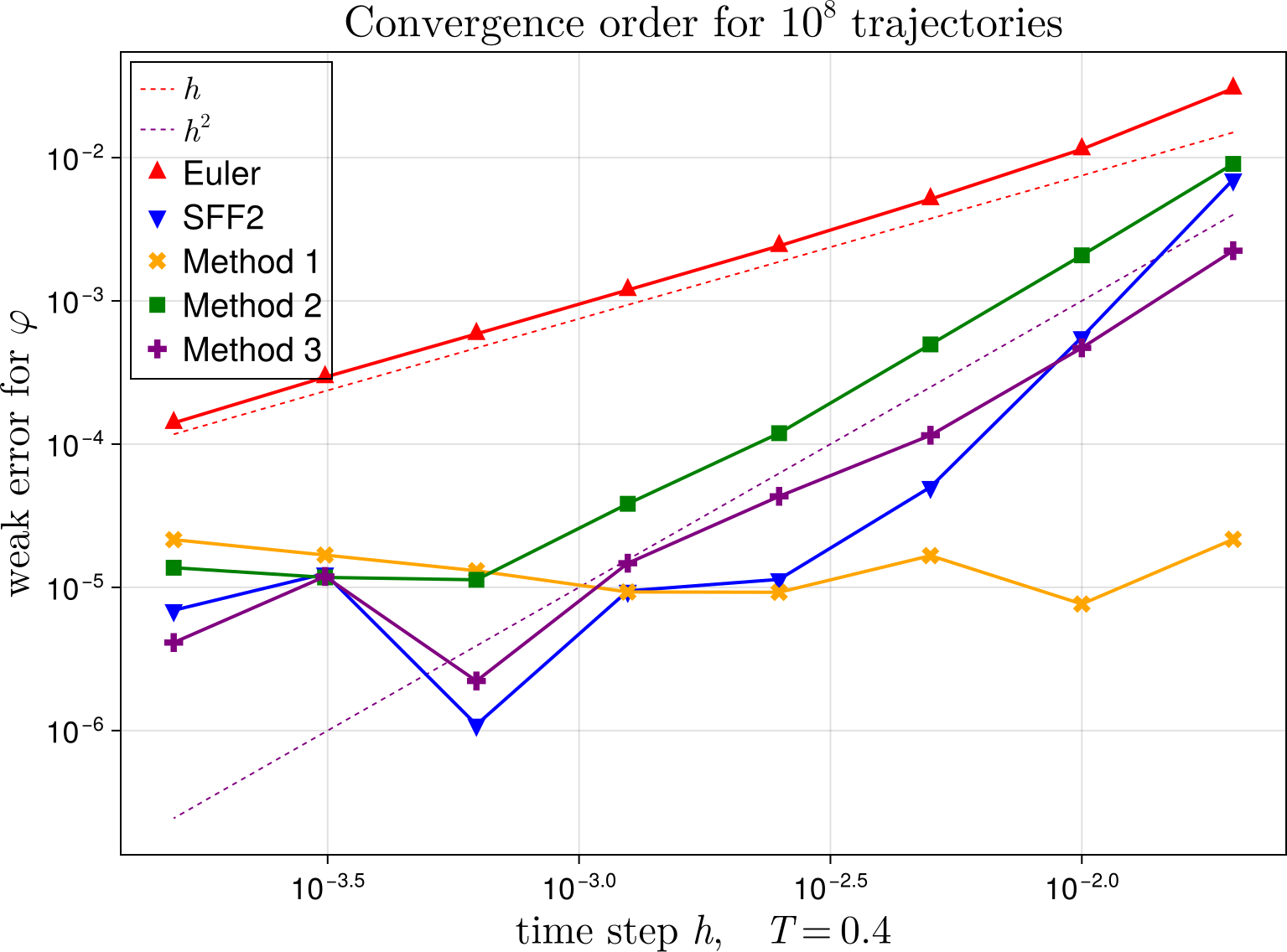}\label{fig:order_sphere}
 \end{subfigure}
 \caption{\scriptsize Trajectory for the Von-Mises Fisher dynamics on the sphere \(\mathds{S}^2\) at time \(T=0.4\) and step time \(h = 4 \times 10^{-6}\)\ (Left). Convergence order for the Von-Mises Fisher dynamics on the sphere \(\mathds{S}^2\) at time \(T=0.4\) with \(10^8\) trajectories and with the test function \(\varphi : (x,y,z) \mapsto z^2 \)\ (Right). }\label{fig:ergodique_sphere}
\end{figure}

\section{Conclusion}\label{section:conclusion}

In this paper, we designed new numerical schemes of high accuracy tailored for the long time sampling of ergodic stochastic dynamics on Riemannian manifolds. We generalized the Euclidean analysis~\cite{Abdulle14hon,Debussche12wbe} to characterize the intrinsic approximation for the invariant measure on manifolds, and presented new algebraic operations on exotic forests~\cite{Bronasco25hoi} to compute the order conditions in long time.
The analysis applies naturally to Riemannian Langevin dynamics and the new methods, which generalise in particular the Leimkuhler-Matthews method, show second order of convergence for a reduced cost on numerical experiments.

The present paper opens several avenues for future research.
The new discretisations could naturally be combined with popular sampling techniques, such as perturbations~\cite{Lelievre13onr, Duncan16vru, Abdulle19act, Bronasco25els}.
The analysis for the invariant measure, as well as the associated algebraic formalism, will be extended for the design of modified equations of arbitrarily high order for sampling SDEs on manifolds and for the challenging derivation of intrinsic stochastic backward error analysis, generalising~\cite{Laurent20eab, Bronasco22cef}.
Similar to the Euclidean setting~\cite{Busnot26osr}, one could extend the analysis to general SDEs with multiplicative noise and the creation of higher order methods.
Our approach follows the one of deterministic Lie-group methods and thus relies on a curvature-free connection, unrelated to the natural Levi-Civita connection on Riemannian manifolds. We will extend the frozen-flow methods so that they rely only on Riemannian operations. The study of such new methods is already open in the deterministic setting and relies on the challenging general understanding of the connection algebra~\cite{AlKaabi22aao, MuntheKaas24gio, Stava24oca}.
Finally, the new algebraic formalism of planar exotic series is interesting in itself and could be studied for its universal combinatorial, algebraic, and geometric properties~\cite{McLachlan16bsm, MuntheKaas16abs, Floystad20tup, Laurent23tue, Busnot26tft}, but also for its potential applications in different fields, in the spirit of the use of Butcher series and their extensions in rough paths~\cite{Hairer15gvn,Lejay22cgr}, renormalisation theory~\cite{Brouder00rkm}, variational calculus~\cite{Laurent23tab,Laurent23tld}, approximation of PDEs~\cite{Bronsard22aod}, mathematical physics~\cite{Bonicelli25ebs}, and wave kinetic dynamics~\cite{Deng23fdo}.
These projects will be studied in upcoming works.

\bigskip

\noindent{\textbf{Acknowledgements}.}\
The authors would like to thank Baptiste Huguet for helpful comments on a previous version of the present work.
The authors acknowledge the support from the French program ANR-25-CE40--2862--01 (MaStoC --- Manifolds and Stochastic Computations).
Experiments presented in this paper were carried out using the Abaca testbed, supported by Inria (see https://abaca.inria.fr).

\footnotesize{
 \bibliographystyle{abbrv} 
 \bibliography{ma_bibliographie}
}

\end{document}

%% file: planarforest_tikz.tex
\newcommand{\forestA}{
    \tikz[planar forest ] {

        \node[b] at (0.0, 0.0) { }
        child {node[b] at (0.0, 1.0) { }
            }
        ;
    }}

\newcommand{\forestB}{
    \tikz[planar forest ] {

        \node[b] at (0.0, 0.0) { }
        child {node[lc] at (-0.5, 1.0) { 1 }
            }
        child {node[lc] at (0.5, 1.0) { 1 }
            }
        ;
    }}

\newcommand{\forestC}{
    \tikz[planar forest ] {

        \node[b] at (0.0, 0.4) { }
        ;
        \node[b] at (1.0, 0.4) { }
        ;
    }}

\newcommand{\forestD}{
    \tikz[planar forest ] {

        \node[b] at (0.0, 0.0) { }
        child {node[lc] at (0.0, 1.0) { 1 }
            }
        ;
        \node[lc] at (1.0, 0.0) { 1 }
        ;
    }}

\newcommand{\forestE}{
    \tikz[planar forest ] {

        \node[lc] at (0.0, 0.0) { 1 }
        ;
        \node[b] at (1.0, 0.0) { }
        child {node[lc] at (0.0, 1.0) { 1 }
            }
        ;
    }}

\newcommand{\forestF}{
    \tikz[planar forest ] {

        \node[b] at (0.0, 0.4) { }
        ;
        \node[lc] at (1.0, 0.4) { 1 }
        ;
        \node[lc] at (2.0, 0.4) { 1 }
        ;
    }}

\newcommand{\forestG}{
    \tikz[planar forest ] {

        \node[lc] at (0.0, 0.4) { 1 }
        ;
        \node[b] at (1.0, 0.4) { }
        ;
        \node[lc] at (2.0, 0.4) { 1 }
        ;
    }}

\newcommand{\forestH}{
    \tikz[planar forest ] {

        \node[lc] at (0.0, 0.4) { 1 }
        ;
        \node[lc] at (1.0, 0.4) { 1 }
        ;
        \node[b] at (2.0, 0.4) { }
        ;
    }}

\newcommand{\forestI}{
    \tikz[planar forest ] {

        \node[lc] at (0.0, 0.4) { 2 }
        ;
        \node[lc] at (1.0, 0.4) { 2 }
        ;
        \node[lc] at (2.0, 0.4) { 1 }
        ;
        \node[lc] at (3.0, 0.4) { 1 }
        ;
    }}

\newcommand{\forestJ}{
    \tikz[planar forest ] {

        \node[lc] at (0.0, 0.4) { 2 }
        ;
        \node[lc] at (1.0, 0.4) { 1 }
        ;
        \node[lc] at (2.0, 0.4) { 2 }
        ;
        \node[lc] at (3.0, 0.4) { 1 }
        ;
    }}

\newcommand{\forestK}{
    \tikz[planar forest ] {

        \node[lc] at (0.0, 0.4) { 1 }
        ;
        \node[lc] at (1.0, 0.4) { 2 }
        ;
        \node[lc] at (2.0, 0.4) { 2 }
        ;
        \node[lc] at (3.0, 0.4) { 1 }
        ;
    }}

\newcommand{\forestL}{
    \tikz[planar forest ] {

        \node[b] at (0.0, 0.4) { }
        ;
    }}

\newcommand{\forestM}{
    \tikz[planar forest ] {

        \node[lc] at (0.0, 0.4) { 1 }
        ;
        \node[lc] at (1.0,0.4) { 1 }
        ;
    }}

\newcommand{\forestN}{
    \tikz[planar forest ] {

        \node[lc] at (0.0, 0.4) { n }
        ;
    }}

\newcommand{\forestAA}{
    \tikz[planar forest ] {
        \node[lc] at (0.0, 0.0) { 2 }
        ;
        \node[lc] at (1.0, 0.0) { 1 }
        ;
        \node[b] at (2.0, 0.0) { }
        child {node[lc] at (0.0, 1.0) { 1 }
            }
        ;\node[b] at (3.0, 0.0) { }
        child {node[lc] at (0.0, 1.0) { 2 }
            }
        ;
    }}

\newcommand{\forestAB}{
    \tikz[planar forest ] {

        \node[lc] at (0.0, 0.0) { 1 }
        ;
        \node[b] at (1.0, 0.0) { }
        child {node[lc] at (-0.5, 1.0) { 2 }
            }
        child {node[lc] at (0.5, 1.0) { 1 } }
        ;\node[b] at (2.5, 0.0) { }
        child {node[lc] at (0.0, 1.0) { 2 } }
        ;
    }}

\newcommand{\forestAC}{
    \tikz[planar forest ] {

        \node[lc] at (0.0, 0.0) { 1 }
        ;
        \node[b] at (1.0, 0.0) { }
        child {node[lc] at (0.0, 1.0) { 1 } }
        ;\node[b] at (2.5, 0.0) { }
        child {node[lc] at (-0.5, 1.0) { 2 } }
        child {node[lc] at (0.5, 1.0) { 2 }
            }
        ;
    }}

\newcommand{\forestAD}{
    \tikz[planar forest ] {

        \node[lc] at (0.0, 0.0) { 1 }
        ;
        \node[b] at (1.0, 0.0) { }
        child {node[lc] at (0.0, 1.0) { 1 } }
        ;\node[b] at (2.0, 0.0) { }
        child {node[b] at (0.0, 1.0) {  } }
        ;
    }}

\newcommand{\forestAE}{
    \tikz[planar forest ] {

        \node[b] at (0.0, 0.0) { }
        child {node[lc] at (-0.5, 1.0) {1 }}
        child {node[lc] at (0.5, 1.0) { 2}}
        ;
        \node[b] at (2.0, 0.0) { }
        child {node[lc] at (-0.5, 1.0) {1 }}
        child {node[lc] at (0.5, 1.0) { 2}};
    }}

\newcommand{\forestAF}{
    \tikz[planar forest ] {

        \node[b] at (0.0, 0.0) { }
        child {node[lc] at (-0.5, 1.0) {1 }}
        child {node[lc] at (0.5, 1.0) { 2}}
        ;
        \node[b] at (2.0, 0.0) { }
        child {node[lc] at (-0.5, 1.0) {2 }}
        child {node[lc] at (0.5, 1.0) { 1}};
    }}

\newcommand{\forestAG}{
    \tikz[planar forest ] {

        \node[b] at (0.0, 0.0) { }
        child {node[lc] at (-0.5, 1.0) {1 }}
        child {node[lc] at (0.5, 1.0) { 1}}
        ;
        \node[b] at (2.0, 0.0) { }
        child {node[lc] at (-0.5, 1.0) {2 }}
        child {node[lc] at (0.5, 1.0) { 2}};
    }}

\newcommand{\forestAH}{
    \tikz[planar forest ] {

        \node[b] at (0.0, 0.0) { }
        child {node[lc] at (0.0, 1.0) {1 }}
        ;
        \node[b] at (2.0, 0.0) { }
        child {node[lc] at (-1.0, 1.0) {1 }}
        child {node[lc] at (0.0, 1.0) { 2}}
                child {node[lc] at (1.0, 1.0) { 2}};
    }}

\newcommand{\forestAI}{
    \tikz[planar forest ] {

        \node[b] at (0.0, 0.0) { }
        child {node[lc] at (-0.5, 1.0) {1 }}
        child {node[lc] at (0.5, 1.0) { 1}}
        ;
        \node[b] at (1.5, 0.0) { }
        child {node[b] at (0., 1.0) { }};
    }}

\newcommand{\forestAJ}{
    \tikz[planar forest ] {

        \node[b] at (0.0, 0.0) { }
        child {node[lc] at (0.0, 1.0) {1 }}
        ;
        \node[b] at (1.5, 0.0) { }
                child {node[lc] at (-0.5, 1.0) { 1}}
        child {node[b] at (0.5, 1.0) { }};
    }}

\newcommand{\forestAK}{
    \tikz[planar forest ] {

        \node[b] at (0.0, 0.0) { }
        child {node[lc] at (-0.5, 1.0) {1 }}
        child {node[b] at (0.5, 1.0) { }}
        ;
        \node[b] at (1.5, 0.0) { }
        child {node[lc] at (0., 1.0) {1 }};
    }}

\newcommand{\forestAL}{
    \tikz[planar forest ] {

        \node[b] at (0.0, 0.0) { }
        child {node[b] at (0.0, 1.0) { }}
        ;
        \node[b] at (1.5, 0.0) { }
                child {node[lc] at (-0.5, 1.0) { 1}}
        child {node[lc] at (0.5, 1.0) { 1}};
    }}

\newcommand{\forestAM}{
    \tikz[planar forest ] {

        \node[b] at (0.0, 0.0) { }
        child {node[b] at (0.0, 1.0) { }}
        ;
        \node[b] at (1.0, 0.0) { }
                child {node[b] at (0., 1.0) { }}
    }}

%% file: first_publication_compiled.bbl
\begin{thebibliography}{10}

\bibitem{Abdulle19act}
A.~Abdulle, G.~A. Pavliotis, and G.~Vilmart.
\newblock Accelerated convergence to equilibrium and reduced asymptotic
  variance for {L}angevin dynamics using {S}tratonovich perturbations.
\newblock {\em C. R. Math. Acad. Sci. Paris}, 357(4):349--354, 2019.

\bibitem{Abdulle14hon}
A.~Abdulle, G.~Vilmart, and K.~C. Zygalakis.
\newblock High order numerical approximation of the invariant measure of
  ergodic {SDE}s.
\newblock {\em SIAM J. Numer. Anal.}, 52(4):1600--1622, 2014.

\bibitem{Abdulle15lta}
A.~Abdulle, G.~Vilmart, and K.~C. Zygalakis.
\newblock Long time accuracy of {L}ie-{T}rotter splitting methods for
  {L}angevin dynamics.
\newblock {\em SIAM J. Numer. Anal.}, 53(1):1--16, 2015.

\bibitem{AlKaabi22aao}
M.~J.~H. Al-Kaabi, K.~Ebrahimi-Fard, D.~Manchon, and H.~Z. Munthe-Kaas.
\newblock Algebraic aspects of connections: From torsion, curvature, and
  post-{L}ie algebras to {G}avrilov's double exponential and special
  polynomials.
\newblock {\em Journal of Noncommutative Geometry}, 19(1):297--335, 2023.

\bibitem{Bronsard22aod}
Y.~Alama~Bronsard, Y.~Bruned, and K.~Schratz.
\newblock Approximations of dispersive {PDE}s in the presence of low-regularity
  randomness.
\newblock {\em Found. Comput. Math.}, pages 1--51, 2024.

\bibitem{Antonyuk07nonexp}
A.~Antonyuk and A.~Antonyuk.
\newblock Nonexplosion and solvability of nonlinear diffusion equations on
  noncompact manifolds.
\newblock {\em Ukr. Math. J.}, 59:1632--1652, 2007.

\bibitem{Bakry86cne}
D.~Bakry.
\newblock Un critère de non-explosion pour certaines diffusions sur une
  variété riemannienne complète.
\newblock {\em C.R. Acad. Sc. Paris}, 303(1):23--26, 1986.

\bibitem{Baxter60aap}
G.~Baxter et~al.
\newblock An analytic problem whose solution follows from a simple algebraic
  identity.
\newblock {\em Pacific J. Math}, 10(3):731--742, 1960.

\bibitem{Bharath23sae}
K.~Bharath, A.~Lewis, A.~Sharma, and M.~V. Tretyakov.
\newblock Sampling and {E}stimation on {M}anifolds using the {L}angevin
  {D}iffusion.
\newblock {\em Journal of Machine Learning Research}, 26(71):1--50, 2025.

\bibitem{Bonicelli25ebs}
A.~Bonicelli.
\newblock Exotic {B}-series representation of the {F}eller semigroup for {I}tô
  diffusions and the {MSR} path integral.
\newblock {\em arXiv preprint arXiv:2510.23102}, 2025.

\bibitem{BouRabee10lra}
N.~Bou-Rabee and H.~Owhadi.
\newblock Long-run accuracy of variational integrators in the stochastic
  context.
\newblock {\em SIAM J. Numer. Anal.}, 48(1):278--297, 2010.

\bibitem{Bronasco22ebs}
E.~Bronasco.
\newblock Exotic {B}-series and {S}-series: algebraic structures and order
  conditions for invariant measure sampling.
\newblock {\em Found. Comput. Math.}, pages 1--31, 2024.

\bibitem{Bronasco22cef}
E.~Bronasco and A.~Busnot~Laurent.
\newblock Hopf algebra structures for the backward error analysis of ergodic
  stochastic differential equations.
\newblock {\em Numer. Math.}, pages 1--61, 2026.

\bibitem{Bronasco25hoi}
E.~Bronasco, A.~Busnot~Laurent, and B.~Huguet.
\newblock High order integration of stochastic dynamics on {R}iemannian
  manifolds with frozen-flow methods.
\newblock {\em arXiv:2503.21855}, 2025.

\bibitem{Bronasco25els}
E.~Bronasco, B.~Leimkuhler, D.~Phillips, and G.~Vilmart.
\newblock Efficient {L}angevin sampling with position-dependent diffusion.
\newblock {\em arXiv:2501.02943}, 2025.

\bibitem{Brouder00rkm}
C.~Brouder.
\newblock Runge--{K}utta methods and renormalization.
\newblock {\em Eur. Phys. J. C}, 12(3):521--534, 2000.

\bibitem{Busnot26osr}
A.~Busnot~Laurent, K.~Debrabant, and A.~Kv{\ae}rn\o.
\newblock Optimal stochastic {R}unge-{K}utta methods for the weak integration
  of stochastic dynamics.
\newblock {\em arXiv:2603.24255}, 2026.

\bibitem{Busnot25pha}
A.~Busnot~Laurent, Y.~Li, and Y.~Sheng.
\newblock Post-{H}opf algebroids, post-{L}ie-{R}inehart algebras and geometric
  numerical integration.
\newblock {\em arXiv:2512.21971}, 2025.

\bibitem{Busnot26tft}
A.~Busnot~Laurent, H.~Munthe-Kaas, and G.~S. Venkatesh.
\newblock The free tracial post-{L}ie-{R}inehart algebra of planar aromatic
  trees for the design of divergence-free {L}ie-group methods.
\newblock {\em arXiv:2603.28437}, 2026.

\bibitem{Butcher72aat}
J.~C. Butcher.
\newblock An algebraic theory of integration methods.
\newblock {\em Math. Comp.}, 26:79--106, 1972.

\bibitem{Butcher21bsa}
J.~C. Butcher.
\newblock {\em B-series: algebraic analysis of numerical methods}.
\newblock Springer, 2021.

\bibitem{Celledoni03cfl}
E.~Celledoni, A.~Marthinsen, and B.~Owren.
\newblock Commutator-free {L}ie group methods.
\newblock {\em Future Generation Computer Systems}, 19(3):341--352, 2003.

\bibitem{Crouch93nio}
P.~E. Crouch and R.~Grossman.
\newblock Numerical integration of ordinary differential equations on
  manifolds.
\newblock {\em Journal of Nonlinear Science}, 3:1--33, 1993.

\bibitem{Debussche12wbe}
A.~Debussche and E.~Faou.
\newblock Weak backward error analysis for {SDE}s.
\newblock {\em SIAM J. Numer. Anal.}, 50(3):1735--1752, 2012.

\bibitem{Deng23fdo}
Y.~Deng and Z.~Hani.
\newblock Full derivation of the wave kinetic equation.
\newblock {\em Inventiones mathematicae}, 233(2):543--724, 2023.

\bibitem{Duncan16vru}
A.~B. Duncan, T.~Leli\`evre, and G.~A. Pavliotis.
\newblock Variance reduction using nonreversible {L}angevin samplers.
\newblock {\em J. Stat. Phys.}, 163(3):457--491, 2016.

\bibitem{Ebrahimi15otl}
K.~Ebrahimi-Fard, A.~Lundervold, and H.~Z. Munthe-Kaas.
\newblock On the {L}ie enveloping algebra of a post-{L}ie algebra.
\newblock {\em J. Lie Theory}, 25(4):1139--1165, 2015.

\bibitem{Ebrahimi14tme}
K.~Ebrahimi-Fard and D.~Manchon.
\newblock The {M}agnus expansion, trees and {K}nuth’s rotation
  correspondence.
\newblock {\em Found. Comput. Math.}, 14(1):1--25, 2014.

\bibitem{Floystad20tup}
G.~Fl{\o}ystad, D.~Manchon, and H.~Z. Munthe-Kaas.
\newblock The universal pre-{L}ie-{R}inehart algebras of aromatic trees.
\newblock In {\em Geometric and harmonic analysis on homogeneous spaces and
  applications}, volume 366 of {\em Springer Proc. Math. Stat.}, pages
  137--159. Springer, Cham, [2021] \copyright 2021.

\bibitem{Grong23pla}
E.~Grong, H.~Z. Munthe-Kaas, and J.~Stava.
\newblock Post-{L}ie algebra structure of manifolds with constant curvature and
  torsion.
\newblock {\em Journal of Lie Theory}, 34(2):339--352, 2024.

\bibitem{Hairer06gni}
E.~Hairer, C.~Lubich, and G.~Wanner.
\newblock {\em Geometric numerical integration}, volume~31 of {\em Springer
  Series in Computational Mathematics}.
\newblock Springer-Verlag, Berlin, second edition, 2006.
\newblock Structure-preserving algorithms for ordinary differential equations.

\bibitem{Hairer15gvn}
M.~Hairer and D.~Kelly.
\newblock Geometric versus non-geometric rough paths.
\newblock {\em Ann. Inst. Henri Poincar\'{e} Probab. Stat.}, 51(1):207--251,
  2015.

\bibitem{Hsu02sao}
E.~P. Hsu.
\newblock {\em Stochastic analysis on manifolds}, volume~38 of {\em Graduate
  Studies in Mathematics}.
\newblock American Mathematical Society, Providence, RI, 2002.

\bibitem{Iserles00lgm}
A.~Iserles, H.~Z. Munthe-Kaas, S.~P. N{\o}rsett, and A.~Zanna.
\newblock Lie-group methods.
\newblock In {\em Acta numerica, 2000}, volume~9 of {\em Acta Numer.}, pages
  215--365. Cambridge Univ. Press, Cambridge, 2000.

\bibitem{Laurent21ata}
A.~Laurent.
\newblock {\em Algebraic Tools and Multiscale Methods for the Numerical
  Integration of Stochastic Evolutionary Problems}.
\newblock PhD thesis, University of Geneva, 2021.

\bibitem{Laurent23tld}
A.~Laurent.
\newblock The {L}ie derivative and {N}oether's theorem on the aromatic
  bicomplex for the study of volume-preserving numerical integrators.
\newblock {\em J. Comput. Dyn.}, 11(1):10--22, 2024.

\bibitem{Laurent23tab}
A.~Laurent, R.~I. McLachlan, H.~Z. Munthe-Kaas, and O.~Verdier.
\newblock The aromatic bicomplex for the description of divergence-free
  aromatic forms and volume-preserving integrators.
\newblock {\em Forum Math. Sigma}, 11:Paper No. e69, 2023.

\bibitem{Laurent23tue}
A.~Laurent and H.~Munthe-Kaas.
\newblock The universal equivariance properties of exotic aromatic {B}-series.
\newblock {\em Found. Comput. Math.}, 25(5):1595--1626, 2025.

\bibitem{Laurent20eab}
A.~Laurent and G.~Vilmart.
\newblock Exotic aromatic {B}-series for the study of long time integrators for
  a class of ergodic {SDE}s.
\newblock {\em Math. Comp.}, 89(321):169--202, 2020.

\bibitem{Laurent21ocf}
A.~Laurent and G.~Vilmart.
\newblock Order conditions for sampling the invariant measure of ergodic
  stochastic differential equations on manifolds.
\newblock {\em Found. Comput. Math.}, 22(3):649--695, 2022.

\bibitem{Lee19irm}
J.~Lee.
\newblock {\em Introduction to Riemannian Manifolds}.
\newblock Graduate Texts in Mathematics. Springer International Publishing,
  2019.

\bibitem{Leimkuhler13rco}
B.~Leimkuhler and C.~Matthews.
\newblock Rational construction of stochastic numerical methods for molecular
  sampling.
\newblock {\em Appl. Math. Res. Express. AMRX}, 2013(1):34--56, 2013.

\bibitem{Leimkuhler16tco}
B.~Leimkuhler, C.~Matthews, and G.~Stoltz.
\newblock The computation of averages from equilibrium and nonequilibrium
  {L}angevin molecular dynamics.
\newblock {\em IMA J. Numer. Anal.}, 36(1):13--79, 2016.

\bibitem{Lejay22cgr}
A.~Lejay.
\newblock Constructing general rough differential equations through flow
  approximations.
\newblock {\em Electron. J. Probab.}, 27:Paper No. 7, 24, 2022.

\bibitem{Lelievre13onr}
T.~Leli\`evre, F.~Nier, and G.~A. Pavliotis.
\newblock Optimal non-reversible linear drift for the convergence to
  equilibrium of a diffusion.
\newblock {\em J. Stat. Phys.}, 152(2):237--274, 2013.

\bibitem{Lelievre10fec}
T.~Leli\`evre, M.~Rousset, and G.~Stoltz.
\newblock {\em Free energy computations}.
\newblock Imperial College Press, London, 2010.
\newblock A mathematical perspective.

\bibitem{Li94stochdiff}
X.-M. Li.
\newblock Stochastic differential equations on noncompact manifolds: moment
  stability and its topological consequences.
\newblock {\em Probab. Theory Relat. Fields}, 100:417--428, 1994.

\bibitem{Li23pha}
Y.~Li, Y.~Sheng, and R.~Tang.
\newblock Post-{H}opf algebras, relative {R}ota--{B}axter operators and
  solutions to the {Y}ang--{B}axter equation.
\newblock {\em Journal of Noncommutative Geometry}, 18(2):605--630, 2023.

\bibitem{Luesink26stf}
E.~Luesink and O.~D. Street.
\newblock Symplectic techniques for stochastic differential equations on
  reductive {L}ie groups with applications to {L}angevin diffusions.
\newblock {\em Journal of Differential Equations}, 458:114034, 2026.

\bibitem{Lundervold11hao}
A.~Lundervold and H.~Munthe-Kaas.
\newblock Hopf algebras of formal diffeomorphisms and numerical integration on
  manifolds.
\newblock In {\em Combinatorics and physics}, volume 539 of {\em Contemp.
  Math.}, pages 295--324. Amer. Math. Soc., Providence, RI, 2011.

\bibitem{Malham08slg}
S.~J.~A. Malham and A.~Wiese.
\newblock Stochastic {L}ie group integrators.
\newblock {\em SIAM J. Sci. Comput.}, 30(2):597--617, 2008.

\bibitem{McLachlan16bsm}
R.~I. McLachlan, K.~Modin, H.~Munthe-Kaas, and O.~Verdier.
\newblock B-series methods are exactly the affine equivariant methods.
\newblock {\em Numer. Math.}, 133(3):599--622, 2016.

\bibitem{Muniz22hso}
M.~Muniz, M.~Ehrhardt, M.~G{\"u}nther, and R.~Winkler.
\newblock Higher strong order methods for linear {I}t{\^o} {SDE}s on matrix
  {L}ie groups.
\newblock {\em BIT Numer. Math.}, 62(4):1095--1119, 2022.

\bibitem{MuntheKaas24gio}
H.~Munthe-Kaas.
\newblock Geometric integration on symmetric spaces.
\newblock {\em J. Comput. Dyn.}, 11(1):43--58, 2024.

\bibitem{MuntheKaas23lat}
H.~Munthe-Kaas and J.~Stava.
\newblock Lie admissible triple algebras: The connection algebra of symmetric
  spaces.
\newblock {\em Submitted}, 2023.

\bibitem{MuntheKaas16abs}
H.~Munthe-Kaas and O.~Verdier.
\newblock Aromatic {B}utcher series.
\newblock {\em Found. Comput. Math.}, 16(1):183--215, 2016.

\bibitem{MuntheKaas13opl}
H.~Z. Munthe-Kaas and A.~Lundervold.
\newblock On post-{L}ie algebras, {L}ie--{B}utcher series and moving frames.
\newblock {\em Found. Comput. Math.}, 13:583--613, 2013.

\bibitem{MuntheKaas08oth}
H.~Z. Munthe-Kaas and W.~M. Wright.
\newblock On the {H}opf algebraic structure of {L}ie group integrators.
\newblock {\em Found. Comput. Math.}, 8(2):227--257, 2008.

\bibitem{ONeill83srg}
B.~O'Neill.
\newblock {\em Semi-Riemannian Geometry With Applications to Relativity}.
\newblock Pure and Applied Mathematics. Academic Press, 1983.

\bibitem{Oudom08otl}
J.-M. Oudom and D.~Guin.
\newblock On the {L}ie enveloping algebra of a pre-{L}ie algebra.
\newblock {\em J. K-Theory}, 2(1):147--167, 2008.

\bibitem{Owren06ocf}
B.~Owren.
\newblock Order conditions for commutator-free {L}ie group methods.
\newblock {\em Journal of Physics A: Mathematical and General}, 39(19):5585,
  2006.

\bibitem{Owren99rkm}
B.~Owren and A.~Marthinsen.
\newblock Runge-{K}utta methods adapted to manifolds and based on rigid frames.
\newblock {\em BIT Numer. Math.}, 39(1):116--142, 1999.

\bibitem{Stava24oca}
J.~Stava.
\newblock {\em On connection algebras of symmetric spaces and reductive
  homogeneous spaces}.
\newblock PhD thesis, University of Bergen, 2024.

\bibitem{Talay90eot}
D.~Talay and L.~Tubaro.
\newblock Expansion of the global error for numerical schemes solving
  stochastic differential equations.
\newblock {\em Stochastic Anal. Appl.}, 8(4):483--509 (1991), 1990.

\bibitem{Vilmart15pif}
G.~Vilmart.
\newblock Postprocessed integrators for the high order integration of ergodic
  {SDE}s.
\newblock {\em SIAM J. Sci. Comput.}, 37(1):A201--A220, 2015.

\end{thebibliography}
